\documentclass[12pt]{amsart}

\usepackage{amsmath,amssymb,amsfonts,amscd,pictexwd,dcpic}

\newtheorem{theorem}{Theorem}[section]
\newtheorem{lemma}[theorem]{Lemma}
\newtheorem{proposition}[theorem]{Proposition}
\newtheorem{definition}[theorem]{Definition}
\newtheorem{corollary}[theorem]{Corollary}

\newtheorem{examples}[theorem]{Examples}

\newtheorem{note}[theorem]{Note}

\newcommand{\N}{{\mathbb N}}
\newcommand{\Z}{{\mathbb Z}}

\newcommand{\id}{\operatorname{id}}
\newcommand{\from}{\colon}

\begin{document}
\title[Cofinite Graphs and their Profinite Completions]{Cofinite Graphs and their Profinite Completions}

\author{Amrita Acharyya}
\address{Department of Mathematics and Statistics\\
University of Toledo, Main Campus\\
Toledo, OH 43606-3390}
\email{Amrita.Acharyya@utoledo.edu}

\author{Jon M. Corson}
\address{Department of Mathematics\\
University of Alabama\\
Tuscaloosa, AL 35487-0350}
\email{jcorson@ua.edu}

\author{Bikash Das}
\address{Department of Mathematics\\
University of North Georgia, Gainesville Campus\\
Oakwood, Ga. 30566}
\email{Bikash.Das@ung.edu}



\begin{abstract}
We generalize the idea of cofinite groups, due to B. Hartley, ~\cite{bH77}. First we define cofinite spaces in general. Then, as a special situation, we study cofinite graphs and their uniform completions. 

The idea of constructing a cofinite graph starts with defining a uniform topological graph $\Gamma$, in an appropriate fashion. We endow abstract graphs with uniformities corresponding to separating filter bases of equivalence relations with finitely many equivalence classes over $\Gamma$. It is established that for any cofinite graph there exists a unique cofinite completion.
\end{abstract}

\maketitle

%%%%%%%%%%%%%%  Introduction

\section{Introduction} \label{s:Intro}

Embedding an algebraic object into a projective limit of well-behaved objects is a frequently used tactic in algebra and number theory. 
\begin{enumerate}
\item If $R$ is any commutative ring and $I$ is an ideal, then the $I$-adic completion of $R$ is the projective limit of the inverse system of quotient rings $R/I^n$, $n\ge0$. 
\item The case of $R=\Z$ and $I=(p)$, where $p$ is a prime, yields the $p$-adic integers. These rings are instances of projective limits of finite rings, and thus are profinite rings.
\item In group theory, any residually  finite group can be embedded in a profinite group (i.e., projective limit of finite groups).
\end{enumerate}

There is a topological approach to producing such projective limits known as {\it completion\/}. By imposing a suitable topology on the object making it into a topological object so that Cauchy sequences or Cauchy nets can be defined and used to construct the completion.  In the case of a residually finite group, Hartley~\cite{bH77} introduced the terminology of cofinite groups.\\

Initially we note that, without some modification, the topological approach used in the classical situations to construct and distinguish various completions breaks down for graphs in general. 
The following easy example illustrates this point. 

\subsubsection*{Example} Let $\Gamma$ be an abstract graph with $V(\Gamma) = \{x\mid x\in \Z\}$, $E(\Gamma) = \left\{e_x\mid x\in V(\Gamma)\setminus\left\{0\right\}\right\}$ with $s(e_x) = x - 1$, if $x>0$, $s(e_x) = x + 1$, if $x<0$, $t(e_x) = x$.\\
For each $N\in\N$, form the finite discrete graphs $\Gamma_N$ where $V(\Gamma_N) = \{-N-1,\cdots,0,\cdots, N+1\}$, $E(\Gamma_N) = \left\{e_x\mid x\in V(\Gamma_N)\setminus \left\{0\right\}\right\}\cup\{e, e^{\prime}\}$ and $s(e_x) = x - 1$, if $x>0$, $s(e_x) = x + 1$, if $x<0$, $t(e_x) = x$, $s(e) = N + 1 = t(e), s(e^{\prime}) = -N - 1 = t(e^{\prime})$.\\

For all $N\in \N\cup\{0\}$, let us now define maps of graphs $q_N\from\Gamma\to\Gamma_N$ via 

$q_N(x) = \begin{cases}

  x & \text{$\left| x\right|\leq N + 1$} \\

  N + 1 & \text{$x\geq N + 2$} \\

  -(N + 1) & \text{$x\leq -(N + 2)$}

\end{cases} q_N(e_x) = \begin{cases}

e_x & \text{$\left| x\right|\leq N + 1$} \\

e & \text{$x\geq N + 2$} \\

e^{\prime} & \text{$x\leq -(N + 2)$}

\end{cases}$\\

Consider the uniformity $\Phi_1$ over $\Gamma$ which is induced by the fundamental system of entourages $R_N = (q_N\times q_N)^{-1}[D(\Gamma_N)]$. Clearly, with respect to $\Phi_1, q_N$ is uniformly continuous for all $N\in \N\cup\{0\}$. Let $\tau_{\Phi_1}$ be the topology induced by $\Phi_1$. If $x\in V(\Gamma)$, then $R_{\left|x\right|}[x] = \{x\}$. Similarly, if $e_x\in E(\Gamma)$, then $R_{\left|x\right|}[e_x] = \{e_x\}$. Hence $\tau_{\Phi_1}$ represents the discrete topology over $\Gamma$.

Now let us define $\varphi_{ij}\from\Gamma_j\to\Gamma_i$ for all $i\leq j \in \N\cup\{0\}\}$ via $\varphi_{ij}(x) = \begin{cases}

x & \text{$\left| x\right|\leq i$} \\

i & \text{$x\geq i +1$} \\

-i & \text{$x\leq -(i + 1)$}

\end{cases}\varphi_{ij}(e_x) = \begin{cases}

e_x & \text{$\left| x\right|\leq i$} \\

e & \text{$x\geq i + 1$} \\

e^{\prime} & \text{$x\leq -(i + 1)$}

\end{cases}$\\

$\varphi_{ij}(e) = e, \varphi_{ij}(e^{\prime}) = e^{\prime}$.\\

Clearly, each $\varphi_{ij}$ is a uniformly continuous map of graphs and for $i = j$, $\varphi_{ii} = \id_{\Gamma_i}$. Also if $i\leq j\leq k$, then $\varphi_{jk}\circ\varphi_{ij} = \varphi_{ik}$. Hence $(\Gamma_i, \varphi_{ij})_{i\leq j\in \N\cup\{0\}}$ forms an inverse system of finite discrete graphs. Then by Theorem~\ref{Existence}, we can deduce that $\widehat{\Gamma} = \varprojlim_{i\in \N\cup\{0\}}\Gamma_i$ is a profinite completion of $\Gamma$.\\

Now let us consider the following graph $\Delta$, with $V(\Delta) = \Z\cup\{-\infty, \infty\}$,\\
$E(\Delta) = \{e_x\mid x\in V(\Gamma)\setminus\{0,\infty, -\infty\}\}\cup\{e, e^{\prime}\}$. The source and target maps are defined as $s(e_x) = x - 1$, if $x>0$, $s(e_x) = x + 1$, if $x<0$, $t(e_x) = x$, $s(e) = \infty = t(e), s(e^{\prime}) = -\infty = t(e^{\prime})$.\\

Let $G_1 = \{x|x>0\}$, $G_2 = \{x|x<0\}$, $G_3 = \{e_x|x>0\}$, $G_4 = \{e_x|x<0\}$, $p_1 = \infty, p_2 = e, p_3 = -\infty, p_4 = e^{\prime}$. Now let us define $\tau$ by the collection of the open sets $O\subseteq \Delta$ such that $O\bigcap(\Delta\setminus[\cup_{i = 1}^4\{p_i\}])$ is open in $(\Delta\setminus[\cup_{i = 1}^4\{p_i\}])$,  and for $p_i\in O$, $[(\Delta\setminus[\cup_{i = 1}^4\{p_i\}])\setminus O]\bigcap G_i$ is finite. Then $\tau$ forms a topology over $\Delta$ and with respect to $\tau, \Delta$ is compact, Hausdorff, totally disconnected and thus a compactification (4 point) of the graph $\Gamma$.

Let us define maps $\theta_N\from\Delta\to\Gamma_N$ via 

$\theta_N(x) =  \begin{cases}

x & \text{$\left| x\right|\leq N + 1$} \\

i & \text{$x\geq N + 2$} \\

-i & \text{$x\leq -(N + 2)$}

\end{cases} \theta_N(e_x) = \begin{cases}

e_x & \text{$\left| x\right|\leq N + 1$} \\

e & \text{$x\geq N + 2$} \\

e^{\prime} & \text{$x\leq -(N + 2)$}

\end{cases}$\\

$\theta_N(\infty) = N + 1, \theta_N(-\infty) = - N - 1, \theta_N(e) = e, \theta_N(e^{\prime}) = e^{\prime}$.\\

Clearly each $\theta_N$ is a uniformly continuous map of graphs. \\

Thus $(\Delta, \theta_N)_{N\in \N\cup\{0\}}$ is compatible with the inverse system\\
$(\Gamma_i, \varphi_{ij})_{i\leq j\in \N\cup\{0\}}$ and thus there exists a uniformly continuous map of graphs $\theta\from\Delta\to\widehat{\Gamma}$ such that for the canonical projection maps\\
$\varphi_N\from\widehat{\Gamma}\to\Gamma_N$ the following diagram commutes for all $i\leq j\in \N\cup\{0\}$:
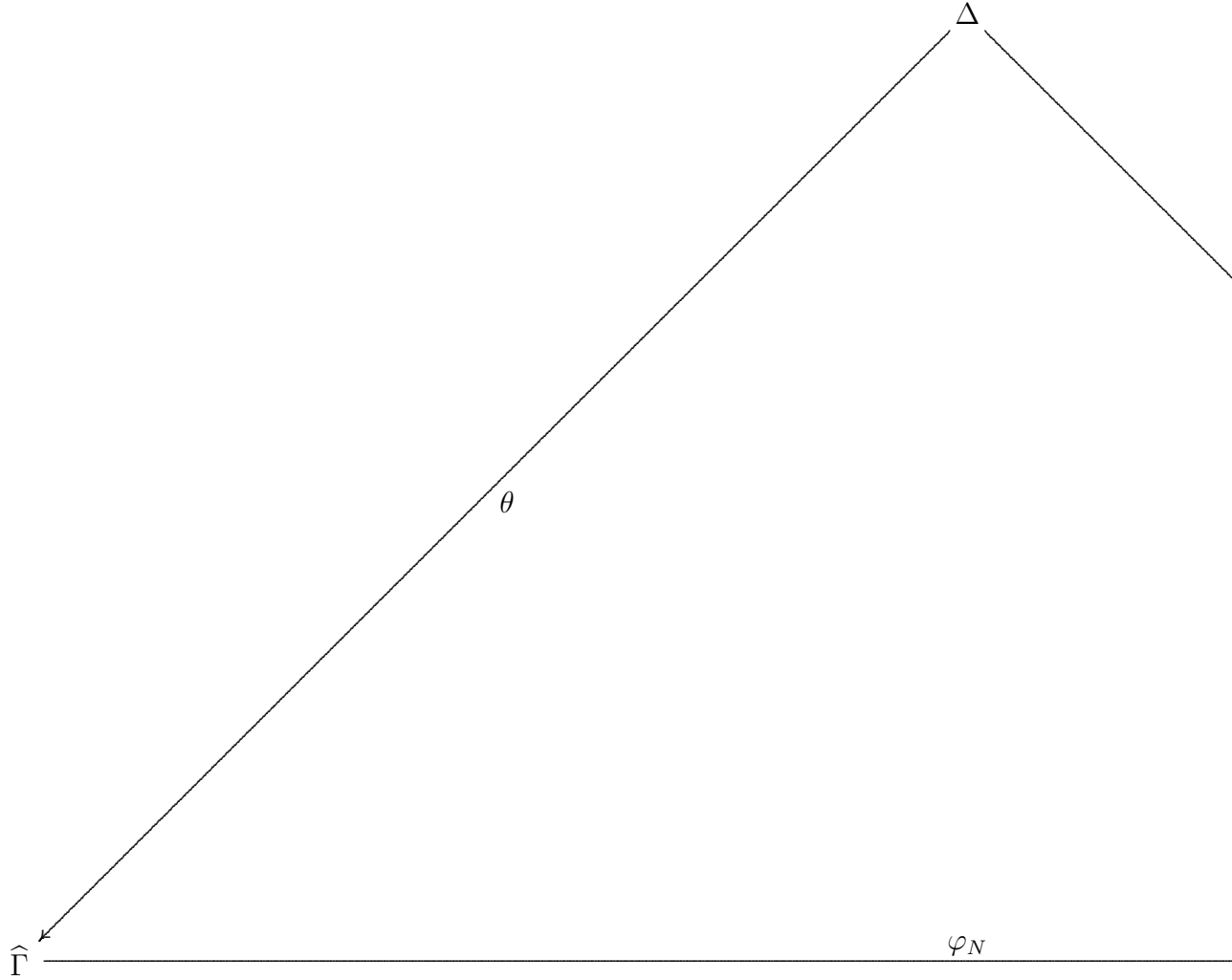
$$\begindc{\commdiag}[25]
\obj(-15,0)[1]{$\widehat{\Gamma}$}
\obj(15,0)[2]{$\Gamma_N$}
\obj(0,15)[3]{$\Delta$}
\mor{1}{2}{$\varphi_N$}
\mor{3}{1}{$\theta$}
\mor{3}{2}{$\theta_N$}
\enddc$$
Since for all $N\in\N\cup\{0\}, \theta_N$ is surjective, $\Gamma_N = \theta_N(\Delta) = \varphi_N(\theta(\Delta))$. Thus $\overline{\theta(\Delta)} = \widehat{\Gamma}$. But since $\Delta$ is compact and $\widehat{\Gamma}$ is Hausdorff, $\theta(\Delta)$ is a closed subset of $\widehat{\Gamma}$ and thus $\theta(\Delta) = \overline{\theta(\Delta)} = \widehat{\Gamma}$. Hence $\theta$ is onto. Also let $\delta_1, \delta_2$ in $\Delta$ be such that $\theta(\delta_1) = \theta(\delta_2)$ and thus $\varphi_N(\theta(\delta_1)) = \varphi_N(\theta(\delta_2))$. Then for all $N\in\N\cup\{0\}$, $\theta_N(\delta_1) = \theta_N(\delta_2)$ and thus $\delta_1 = \delta_2$. Hence $\theta$ is one one and thus $\theta$ is a continuous bijection from a compact space $\Delta$ to a Hausdorff space $\widehat{\Gamma}$ and thus a homeomorphism. Hence $\Delta$ is the cofinite completion of $\Gamma$.

Let us now define $\Sigma_N$, with $V(\Sigma_N) = \{-N,\cdots, 0,\cdots, N, N+1\}$, $E(\Sigma_N) = \{e_x\mid x\in V(\Sigma_N)\setminus \{0\}\}\cup\{e\}$and $s(e_x) = x - 1$, if $x > 0$, $s(e_x) = x + 1$, if $x < 0$, $t(e_x) = x$, if $-N\leq x\leq N+ 1$, $t(e_{-(N + 1)}) = N + 1$, $s(e) = N + 1 = t(e)$.\\

Let us now define map of graphs $q^{\prime}_N\from\Gamma\to\Sigma_N$ via 

$q^{\prime}_N(x) = \begin{cases}

  x & \text{$\left| x\right|\leq N$} \\

  N + 1 & \text{$\left| x\right|\geq N + 1$} 

 \end{cases} q_N^{\prime}(e_x) = \begin{cases}

e_x & \text{$\left| x\right|\leq N + 1$} \\

e & \text{$\left| x\right|\geq N + 2$} 

\end{cases}$

Consider the uniformity $\Phi_2$ over $\Gamma$, consisting of the entourages $S_N = (q^{\prime}_N\times q^{\prime}_N)^{-1}[D(\Sigma_N)]$. Clearly, with respect to $\Phi_2, q^{\prime}_N$ is uniformly continuous for all $N\in \N\cup\{0\}$. Let $\tau_{\Phi_2}$ be the topology induced by $\Phi_2$. Now if $x\in V(\Gamma)$, then $R_{\left|x\right|}[x] = \{x\}$, similarly, if $e_x\in E(\Gamma)$, $R_{\left|x\right|}[e_x] = \{e_x\}$. Hence $\tau_{\Phi_2}$ represents the discrete topology over $\Gamma$ too.

Now let us define $\psi_{ij}\from\Sigma_j\to\Sigma_i$ for all $i\leq j \in \N\cup\{0\}\}$ as follows: 

$\psi_{ij}(x) = \begin{cases}

x & \text{$\left| x\right|\leq i$} \\

i & \text{$\left| x\right|\geq i + 1$} 

\end{cases} \psi_{ij}(e_x) = \begin{cases}

e_x & \text{$\left| x\right|\leq i + 1$} \\

e & \text{$\left| x\right|\geq i + 2$}\\

\end{cases}$\\

 $\psi_{ij}(e) = e$. \\

Clearly, each $\psi_{ij}$ is a uniformly continuous map of graphs and for $i = j$, $\psi_{ii} = \id_{\Sigma_i}$, $i\leq j\leq k$, $\psi_{jk}\circ\psi_{ij} = \psi_{ik}$. Hence $(\Sigma_i, \psi_{ij})_{i\leq j\in \N\cup\{0\}}$ forms an inverse system of finite discrete graphs. Then by Theorem~\ref{Existence}, we deduce that $\widehat{\Sigma} = \varprojlim_{i\in \N\cup\{0\}}\Sigma_i$ is a profinite completion of $\Gamma$. 

Finally, let us now consider the graph $\Delta^{\prime}$, with $V(\Delta^{\prime}) = \Z\cup\{\infty\}$,\\
$E(\Delta^{\prime}) = \{e_x\mid x\in V(\Delta^{\prime})\setminus\{0,\infty\}\}\cup\{e\}$ with $s(e_x) = x - 1$, if $x>0$, $s(e_x) = x + 1$, if $x<0$, $t(e_x) = x$, $s(e) = \infty = t(e)$. Now let $G_1 = \{x\mid x \in Z\}, G_2 = \{e_x\mid x \in \Z\setminus\{0\}\}, p_1 = \infty, p_2 = e$. Now let us define $\tau^{\prime}$ by the collection of the open sets $O\subseteq \Delta$ such that $O\bigcap(\Delta^\prime\setminus[\cup_{i = 1}^2\{p_i\}])$ is open in $(\Delta^\prime\setminus[\cup_{i = 1}^2\{p_i\}])$,  and for $p_i\in O$, $[(\Delta^\prime\setminus[\cup_{i = 1}^4\{p_i\}])\setminus O]\bigcap G_i$ is finite. Then $\tau$ forms a topology over $\Delta^\prime$ and with respect to $\tau^\prime, \Delta^\prime$ is compact, Hausdorff, totally disconnected and thus a compactification (2 point) of the graph $\Gamma$.\\

Let us define maps $\zeta_N\from\Delta^{\prime}\to\Sigma_N$ via 

$\zeta_N(x) = \begin{cases}

x & \text{$\left| x\right|\leq N$} \\

e & \text{$\left| x\right|\geq N + 1$} 

\end{cases} \zeta_N(e_x) = \begin{cases}

e_x & \text{$\left| x\right|\leq N + 1$} \\

e & \text{$\left| x\right|\geq N + 2$}\\

\end{cases}$\\

$\zeta_N(\infty) = N + 1$, $\zeta_N(e) = e$.\\

Clearly each $\zeta_N$ is a uniformly continuous map of graphs.\\

So $(\Delta^{\prime}, \zeta_N)_{N\in \N\cup\{0\}}$ is compatible with $(\Sigma_i, \psi_{ij})_{i\leq j\in \N\cup\{0\}}$ and thus there exists a uniformly continuous map of graphs $\zeta\from\Delta^{\prime}\to\widehat{\Sigma}$ such that for the canonical projection maps $\psi_N\from\widehat{\Sigma}\to\Sigma_N$ the following diagram commutes for all $i\leq j\in \N\cup\{0\}$:
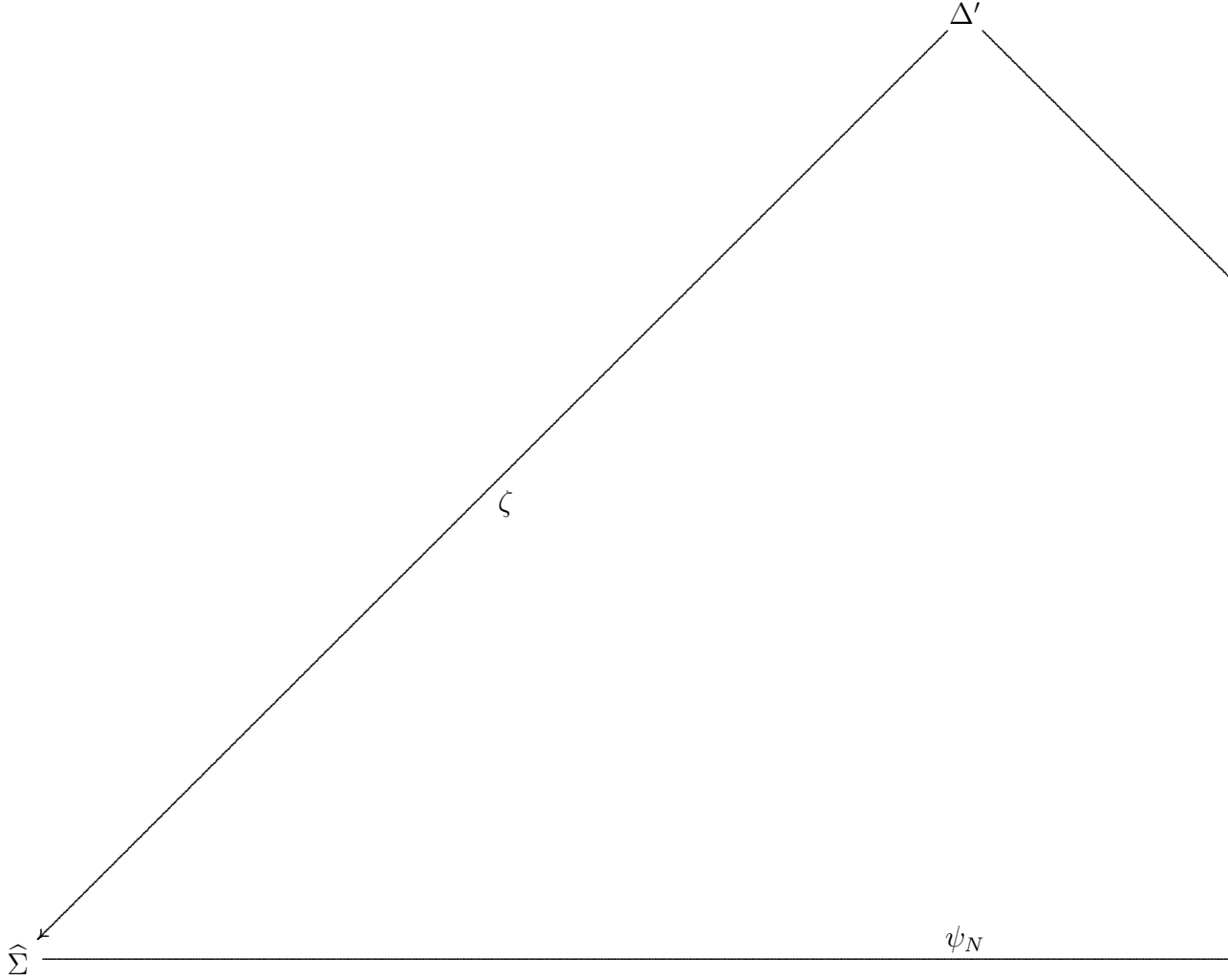
$$\begindc{\commdiag}[25]
\obj(-15,0)[1]{$\widehat{\Sigma}$}
\obj(15,0)[2]{$\Sigma_N$}
\obj(0,15)[3]{$\Delta^{\prime}$}
\mor{1}{2}{$\psi_N$}
\mor{3}{1}{$\zeta$}
\mor{3}{2}{$\zeta_N$}
\enddc$$
Since for all $N\in\N\cup\{0\}, \zeta_N$ is surjective, $\Sigma_N = \zeta_N(\Delta^{\prime}) = \psi_N(\zeta(\Delta^{\prime}))$. Thus $\overline{\zeta(\Delta^{\prime})} = \widehat{\Sigma}$. But since $\Delta^{\prime}$ is compact and $\widehat{\Sigma}$ is Hausdorff, $\zeta(\Delta^{\prime})$ is a closed subset of $\widehat{\Sigma}$ and thus $\zeta(\Delta^{\prime}) = \overline{\zeta(\Delta^{\prime})} = \widehat{\Sigma}$. Hence $\zeta$ is onto. Also let $\delta^{\prime}_1, \delta^{\prime}_2\in\Delta^{\prime}$ be such that $\zeta(\delta^{\prime}_1) = \zeta(\delta^{\prime}_2)$ and thus $\psi_N(\zeta(\delta^{\prime}_1)) = \psi_N(\zeta(\delta^{\prime}_2))$. Then for all $N\in\N\cup\{0\}$, $\zeta_N(\delta^{\prime}_1) = \zeta_N(\delta^{\prime}_2)$ and thus $\delta^{\prime}_1 = \delta^{\prime}_2$. Hence $\zeta$ is one one and thus $\zeta$ is a continuous bijection from a compact space $\Delta^{\prime}$ to a Hausdorff space $\widehat{\Sigma}$ and thus a homeomorphism. Hence $\Delta^{\prime}$ is the cofinite completion of $\Gamma$.

But $\Delta$ is not isomorphic to $\Delta^{\prime}$ as they are the two point and 4 point compactifications for $\Gamma$ respectively. So the example indicates that different uniformities that induces the same topology on a graph can lead us to two non isomorphic completions. 

%%%%%%%%%%%%%%  Preliminaries

\section{Preliminaries} \label{s:Preliminaries}

%%%%%%%%%%%%%%%%  Binary relations

\subsection{Binary relations} \label{s:binrel}

Let $X$ and $Y$ be sets and let $R\subseteq X\times Y$. 
Such a subset $R$ is called a {\it binary relation\/} from $X$ to $Y$.  
For any $x\in X$, we write 
$R[x]=\{y\in Y\mid (x,y)\in R\}$. More generally, for any subset $A$ of $X$, let 
$R[A]=\bigcup\{R[a] \mid a\in A\}$.

The {\it inverse\/} of a binary relation $R\subseteq X\times Y$ is the binary relation 
$R^{-1}\subseteq Y\times X$ given by $R^{-1}=\{(y,x)\mid (x,y)\in R\}$. 
The {\it composition\/} of binary relations $R\subseteq X\times Y$ 
and $S\subseteq Y\times Z$  is the binary relation $SR\subseteq X\times Z$ given by 
$$
SR=\{ (x,z) 
\mid\hbox{there exists $y\in Y$ such that $(x,y)\in R$ and $(y,z)\in S$} \}.
$$
The {\it diagonal\/} $D(X)=\{(x,x)\mid x\in X\}$ in $X\times X$ is the ``equality'' binary relation on $X$.  
For any relation  $R\subseteq X\times Y$, note that the compositions $R\,D(X)=R$ and $D(Y)\,R=R$.

\begin{note}
Composition of binary relations is an associative operation: 
if $R_i\subseteq X_i\times X_{i+1}$ is a binary relation for $i=1,2,3$ then 
$$ 
(R_3R_2)R_1=R_3(R_2R_1).
$$
\end{note}
\begin{note}
Let $R\subseteq X\times Y$ and $S\subseteq Y\times Z$ be binary relations. Then
\begin{enumerate}
\item $(SR)^{-1}=R^{-1}S^{-1}$.
\item for any subset $A$ of $X$, we have that $(SR)[A]=S[R[A]]$.
\end{enumerate}
\end{note}
Let $T_i\subseteq X_i\times Y_i$ be a binary relation for $i=1,2$.  
Then we denote by $T_1\times T_2$ the binary relation from $X_1\times X_2$ to 
$Y_1\times Y_2$ consisting of all pairs $((x_1,x_2),(y_1,y_2))$ such that $(x_i,y_i)\in T_i$ for $i=1,2$. 

\begin{note}\label{rel cart pro}
If $R\subseteq X_1\times X_2$ is a binary relation, then $S = (T_1\times T_2)[R]$ is a binary relation from $Y_1$ to $Y_2$ and 
$$
S=(T_1\times T_2)[R]
=\bigcup\{T_1[x_1]\times T_2[x_2]\mid (x_1,x_2)\in R\}
=T_2RT_1^{-1}
$$
$$\begin{CD}
X_1 @>R>> X_2 \\
@V{T_1}VV @VV{T_2}V \\
Y_1 @>S>> Y_2 
\end{CD}$$ 
\end{note}

\subsection*{Equivalence relations}\label{ss:equiv}

Let $X$ be a set. A {\it binary relation\/} on $X$ is a subset $R\subseteq X\times X$. 
A binary relation $R$ on $X$ is called an {\it equivalence relation\/} if it satisfies three properties:
\begin{enumerate}
\item Reflexive: $R$ contains the diagonal $D(X)=\{(x,x)\mid x\in X\}$.
\item Symmetric: $R^{-1}=R$.
\item Transitive: $R^2\subseteq R$.
\end{enumerate}
It follows that if $R$ is an equivalence relation, then $R^2=R$. To see this suppose $(x,y)\in R$. Then this implies that $(x,y), (y,y)\in R$ by reflexivity of $R$ and thus $(x,y)\in R^2$ by transitivity of $R$. 

\begin{note} \label{n:intersection}
Let $(R_i \mid i\in I)$ be a family of equivalence relations on a set $X$. Then the intersection
$\bigcap_{i\in I}R_i$ is also an equivalence relation on $X$.
\end{note}

It follows that every relation $S$ on $X$ is contained in a unique smallest equivalence relation---namely, the intersection of all equivalence relations that contain $S$. 
We denote it by $\langle S\rangle$ and call it the {\it equivalence relation generated by\/} $S$. 

\begin{note} \label{n:composition}
Let $R_1$ and $R_2$ be equivalence relations on a  set $X$. Then $R_1R_2$ is an equivalence relation if and only if $R_2R_1=R_1R_2$. 
In this case, $R_2R_1=\langle R_1\cup R_2\rangle=R_1R_2$.
\end{note}

\begin{note}[Modular Law]
Let $R$, $R_1$, and $R_2$ be equivalence relations on a set $X$ such that $R\subseteq R_1$. Then $R(R_1\cap R_2)=R_1\cap RR_2$.
\end{note}

\begin{note} Let $X, Y$ be two sets and $f\from X\to Y$ be a function of sets. Let $f = \{(x,y)\in X\times Y\mid f(x) = y\}$ and $f^{-1} = \{(y,x)\in Y\times X\mid f(x) = y\}$. Then the following properties are true,
\begin{enumerate}
\item $(f\times f)[S]=fSf^{-1}$, for all relations $S$ on $X$ and $(f\times f)^{-1}[R]=f^{-1}Rf$, for all relations $R$ over $Y$. (This is a particular case of \ref{rel cart pro}) 
\item The realtion $K_f = \{(x_1,x_2)\in X\times X\mid f(x_1) = f(x_2)\} = (f^{-1}) f = (f\times f)^{-1}[D(Y)]$ is an equivalence relation.
\item $f(f^{-1})\subseteq D(Y)$
\end{enumerate}
\end{note}

\begin{theorem}[Correspondence Theorem]\label{correspondence}
If $f\from X\to Y$ is a map, then $(f\times f)^{-1}[R]$ is an equivalence relation on $X$, for all equivalence relations $R$ over $Y$ and if $f$ is a surjection, also $(f\times f)[S]$ is an equivalence relation on $Y$, for all equivalence relations $S$ over $X$, that contain $K_f$. Moreover if $R$ and $S$ have finitely many equivalence classes in $Y$ and $X$ respectively  then $(f\times f)^{-1}[R]$ and $(f\times f)[S]$ have finitely many equivalence classes, in $X$ and $Y$, respectively. 
\end{theorem} 

\subsection{Cofinite equivalence relations on topological spaces}

Note that if $R$ is an equivalence relation on a set $X$, then 
$$
R=R^3={\textstyle\bigcup}\{ R[a]\times R[b]\mid (a,b)\in R\}.
$$
\begin{proof}
$R^3 = R^2R =RR=R^2=R$, as $R$ is an equivalence relation.
Thus ${\textstyle\bigcup}\{ R[a]\times R[b]\mid (a,b)\in R\} =  RRR^{-1} = RRR = R^3$.
\end{proof}
For topological spaces, this leads to the following observation.

\begin{lemma} \label{n:open}
Let $R$ be an equivalence relation on a topological space $X$. Then $R$ is an open subset of the product space $X\times X$ if and only if $R[a]$ is an open subset of $X$, for each $a$ in $X$. 
\end{lemma}

Notice that in the situation of \ref{n:open}, the quotient space $X/R$ has the discrete topology.  Hence, we will refer to such an equivalence relation $R$ as being {\it co-discrete}. 
It should be  noted that the term 'open' for an equivalence relation on a topological space $X$ typically means something else, namely that the quotient map $X\to X/R$ is an open mapping. 

\begin{definition}[Cofinite equivalence relation]\rm
Let $X$ be a topological space. A {\it cofinite equivalence relation\/} on $X$ is an equivalence relation $R$ such that the quotient space $X/R$ is a finite discrete space.
\end{definition}

In other words, an equivalence relation $R$ on $X$ is cofinite if and only if $R$ is co-discrete and there are only finitely many equivalence classes of $X$ modulo $R$. 

\begin{lemma} \label{n:Properties of cofinite equiv rel}
Cofinite equivalence relations on topological spaces satisfy the following elementary properties:
\begin{enumerate}
\item The intersection $R_1\cap R_2$ of two cofinite equivalence relations $R_1$, $R_2$ on a space $X$ is also cofinite.
\item Let $S$ be an equivalence relation on a space $X$. If $S$ contains a cofinite equivalence relation, then $S$ itself is cofinite.
\item If $R_1$, $R_2$ are commuting equivalence relations on a space $X$, and if one of $R_1$, $R_2$ is cofinite, then the product $R_1R_2$ is also a cofinite equivalence relation.
\item If $f\from X\to Y$ is a continuous map of topological spaces and $R$ is a cofinite equivalence relation on $Y$, then $(f\times f)^{-1}[R]$ is a cofinite equivalence relation on $X$.
\item If $A$ is a subspace of a topological space $X$ and $R$ is a cofinite equivalence relation on $X$, then the restriction $R\cap(A\times A)$ is a cofinite equivalence relation on $A$. 
\item If $X$ is compact, then every co-discrete equivalence relation on $X$ is cofinite. 
\end{enumerate}
\end{lemma}

%%%%%%%%%%%%%%%  Cofinite Spaces

\section{Cofinite Spaces}

We now turn our attention to uniform spaces. Unless otherwise stated, the topology on a uniform space will always be the one induced by its uniformity.

Let $X$ be a uniform space. By a {\it cofinite entourage\/} on $X$ we will mean an entourage $R$ which is also a cofinite equivalence relation on~$X$. As consequences of \ref{n:Properties of cofinite equiv rel}, we see that cofinite entourages satisfy the following elementary properties:

\begin{lemma} \label{n:Properties of cofinite entourages}
Let $X$ and $Y$ be uniform spaces.
\begin{enumerate}
\item The intersection $R_1\cap R_2$ of two cofinite entourages $R_1$, $R_2$ of $X$ is also a cofinite entourage of $X$.
\item Let $S$ be an equivalence relation on $X$. If $S$ contains a cofinite entourage, then $S$ itself is a cofinite entourage.
\item If $R_1$, $R_2$ are commuting equivalence relations on a space $X$, and if one of $R_1$, $R_2$ is a cofinite entourage, then the product $R_1R_2$ is also a cofinite entourage.
\item If $f\from X\to Y$ is a uniformly continuous map and $R$ is a cofinite entourage of $Y$, then $(f\times f)^{-1}[R]$ is a cofinite entourage of $X$.
\item If $X$ is compact and Hausdorff, then every co-discrete equivalence relation on $X$ is a cofinite entourage. 
\end{enumerate}
\end{lemma}

\begin{definition}[Cofinite space]\rm
A {\it cofinite $($uniform$)$ space\/} is a uniform space $X$ whose cofinite entourages form a fundamental system of entourages (i.e., every entourage of $X$ contains a cofinite entourage). 
\end{definition}

\begin{lemma}
For a cofinite space $X$ with a fundamental system of cofinite entourages, say, $I$, the set  $\beta = \{R[x]\mid x\in X, R\in I\}$ forms the basis of the corresponding uniform topology and each $R[x]$ is clopen.
\end{lemma}
\begin{examples}\rm
1. Let $G$ be a cofinite group, i.e., a Hausdorff topological group in which the set of all open normal subgroups of finite index forms a neighborhood base of the identity $1\in G$. 
Then for each open normal subgroup $N$ of $G$, the subset 
$R_N=\{(a,b)\in G\times G\mid ab^{-1}\in N\}$ is a cofinite equivalence relation on $G$. Furthermore, the set $I=\{R_N\mid\hbox{$N$ is an open normal subgroup of $G$}\}$ is a fundamental system of entourages for a uniformity on $G$ that induces its topology. 
In this way, we view $G$ as a cofinite space. 

2. Let $X$ be a compact Hausdorff totally disconnected space. Then, endowed with the unique uniform structure compatible with its topology, $X$  is a cofinite space.  

3. Let $X$ be any set and let $I$ be a separating filter base of equivalence relations on~$X$, each of which has only finitely many equivalence classes. By this we mean that $I$ is a set of equivalence relations that have only finitely many equivalence classes satisfying the two conditions: 
\begin{enumerate}
\item[(i)] If $R_1,R_2\in I$, then there exists $R_3\in I$ such that $R_3\subseteq R_1\cap R_2$.
\item[(ii)] The intersection of all members of $I$ is the diagonal $D(X)$.
\end{enumerate}

Then $I$ is a fundamental system of entourages for a uniform structure making $X$ into a Hausdorff cofinite space. 
\end{examples}

Cofinite spaces have the following elementary properties:

This following lemma is an analogue to similar works done in ~\cite{bH77}, but in the category of general cofinite spaces.
\begin{lemma} \label{n:Cofinite Properties}
Let $X$ be a cofinite space and let $I$ be a fundamental system of cofinite entourages of $X$. Then the following properties hold:
\begin{enumerate}
\item If $W\subseteq X\times X$, then $\overline W =\bigcap_{R\in I} (R\times R)[W] = \bigcap_{R\in I}RWR$ and each $RWR$ is a clopen neighborhood of $W$ in $X\times X$.
\item If $A\subseteq X$, then $\overline A = \bigcup _{R\in I}R[A]$ and each $R[A]$ is a clopen neighborhood of $A$ in $X$.
\end{enumerate} 
\end{lemma}

 \begin{lemma} Let $X$ be a cofinite space and let $I$ be a fundamental system of cofinite entourages of $X$. Then the following statements are equivalent:
\begin{enumerate}
\item $X$ is Hausdorff;
\item $X$ is totally disconnected;
\item $\bigcap_{R\in I}R = D(X)$;
\end{enumerate}
\end{lemma}

Next we consider a general process for constructing cofinite spaces, using what is called by~N.~Bourbaki, ~\cite{nB66}, "initial uniformities". 

\begin{definition}
Let $X$ be a set, let $(X_i)_{i\in I}$ be a family of sets, 
and let $F = (f_i\from X \to X_i)_{i\in I}$ be a family of functions for $X$. We call $F$ a separating family of maps if for all $x\neq y$ in $X$, then exists $i\in I$, such that $f_i(x)\neq f_i(y)$ in $X_i$.
\end{definition}
\begin{proposition} \label{p:Initial uniformity}
Let $X$ be a set, let $(X_i)_{i\in I}$ be a family of cofinite spaces, 
and let $(f_i\from X \to X_i)_{i\in I}$ be a family of functions for $X$.
Let $\bf S$ be the set of all equivalence relations on $X$ of the form 
$(f_i\times f_i)^{-1}[R_i]$, where $i\in I$ and $R_i$ runs through a fundamental system of cofinite entourages of $X_i$.
Finally, let $\bf B$ be the set of all finite intersections of members of $\bf S$. 
Then $\bf B$ is a fundamental system of entourages of a uniformity on $X$ which is the coarsest uniformity on $X$ for which all the mappings $f_i$ are uniformly continuous. 
Endowed with this uniform structure, $X$ becomes a cofinite space. Moreover if each $X_i$ is Hausdorff and $(f_i\from X \to X_i)_{i\in I}$ is a separating family of functions for $X$, then $X$ is Hausdorff as well. 
\end{proposition}

Here are two corollaries of this construction. Let $X$ be as in Proposition~\ref{p:Initial uniformity}.

\begin{corollary} \label{c:Initial uniformity}
If $h\from Y\to X$ is a mapping from a uniform space $Y$, then $h$ is uniformly continuous if and only if each mapping $f_i\circ h\from Y\to X_i$ is uniformly continuous.
\end{corollary}

\begin{corollary} \label{c:Initial topology}
The topology on $X$ induced by the above uniformity is the coarsest topology for which the $f_i$ are continuous.
\end{corollary}

%%%%%%%%%%%%%  Uniform subspaces of cofinite spaces

\subsection{Uniform subspaces of cofinite spaces}

Recall that a {\it uniform subspace\/} of a uniform space $X$ is a subset $A$, endowed with the coarsest uniformity for which the inclusion mapping $A\to X$ is uniformly continuous. This uniformity is called the {\it uniformity induced on\/} $A$ by that of $X$. 

In the case of a uniform subspace, Proposition~\ref{p:Initial uniformity}, can be stated as follows. 

\begin{proposition} \label{n:uniform subspace is cofinite}
Let $A$ be a uniform subspace of a cofinite space $X$. Then the family of all sets of the form $R\cap(A\times A)$, where $R$ runs through a fundamental system of cofinite entourages of $X$, is a fundamental system of entourages of $A$. In particular, $A$ is a cofinite space.
\end{proposition}

By Corollary~\ref{c:Initial topology}, we see that the topology induced on a uniform subspace $A$ of a cofinite space $X$ by its uniformity is the same as the subspace topology on $A$. Recall that the subspace topology on $A$ is the coarsest topology on $A$ such that $i$ is continuous. 
Furthermore, we next observe that restrictions of uniformly continuous maps to uniform subspaces are uniformly continuous.

\begin{proposition} \label{n:restriction to subspace}
Let $f\from X\to Y$ be a uniformly continuous map of cofinite spaces and let $A$, $B$ be uniform subspaces of $X$, $Y$ such that $f(A)\subseteq B$. Then the restriction $f|_A\from A\to B$ is also uniformly continuous.
\end{proposition}
\
%%%%%%%%%%%%%  Products of cofinite spaces

\subsection{Products of cofinite spaces}

Recall that if $(X_i)_{i\in I}$ is a family of uniform spaces, then the coarsest uniformity on the Cartesian product 
$$
X=\prod_{i\in I}X_i
$$
for which the projections $\pi_i\from X\to X_i$ are uniformly continuous is called the
{\it product uniformity}. The set $X$ together with its product uniformity is called the  
{\it product uniform space\/} of this family.

In the case of a Cartesian product of cofinite spaces, Proposition~\ref{p:Initial uniformity} yields the follow result. 

\begin{proposition}
If $X$ is the product uniform space of a family $(X_i)_{i\in I}$ of cofinite spaces, then $X$ is a cofinite space.
\end{proposition}

By Corollary~\ref{c:Initial topology}, the topology induced on a product uniform space 
$X=\prod_{i\in I}X_i$ of a family of cofinite spaces is the same as the product topology on $X$ Recall that the product topology on $X$ is the coarsest topology on $X$ such that each projection $\pi_i$, for all $i\in I$ is continuous. In this situation, Corollary~\ref{c:Initial uniformity} says:
if $f$ is a function from a uniform space $Y$ into the product uniform space $X$, then $f$ is uniformly continuous if and only if the coordinate functions $f_i=\pi_i\circ f$ are uniformly continuous.

%%%%%%%%  Inverse limits of cofinite spaces

\subsection{Inverse limits of cofinite spaces} \label{ss:Inverse limits of cofinite spaces}

Let $(X_i,\phi_{ij})$ be an inverse system of sets indexed by a directed set $I$. 
We say that $(X_i,\phi_{ij})$ is an {\it inverse system of uniform spaces\/} if (i)~each $X_i$ is a uniform space, and (ii)~for all $i\le j$, $\phi_{ij}\from X_j\to X_i$ is uniformly continuous.
The set $X=\varprojlim X_i$ endowed with the coarsest uniformity for which the canonical maps $\phi_i\from X\to X_i$ are uniformly continuous is called the {\it inverse limit of the inverse system of uniform spaces}.

Equivalently, the inverse limit of an inverse system of uniform spaces $(X_i,\phi_{ij})$ is the uniform subspace of the product uniform space $\prod_{i\in I}X_i$ consisting of all points $x$ such that
$$
\pi_i(x)=\phi_{ij}(\pi_j(x))
$$
whenever $i\le j$, and $\pi_i$ is the regular projection map. Also the induced topology on the uniform space $X=\varprojlim X_i$ is the same as the inverse limit of the topologies on the $X_i$; 
see~\cite[Chapter II, \S 2, no.~7]{nB66}. 

In the case of an inverse system of cofinite spaces, Proposition~\ref{p:Initial uniformity} can be stated as follows: 

\begin{proposition} \label{p:Inverse limit of cofinite spaces}
Let $(X_i,\phi_{ij})$ be an inverse system of cofinite spaces and let $X=\varprojlim X_i$ be the inverse limit. For each $i\in I$, let $\phi_i\from X\to X_i$ be the canonical map. Then the collection of all sets $(\phi_i\times\phi_i)^{-1}[R_i]$, where $i$ runs through $I$ and $R_i$ runs through a fundamental system of cofinite entourages of $X_i$, is a fundamental system of cofinite entourages of $X$. In particular, $X$ is a cofinite space.
\end{proposition}

As in any category, inverse limits of cofinite spaces are characterized by a universal property: 
Let $(X_i,\phi_{ij})$ be an inverse system of cofinite spaces, 
let $Y$ be a cofinite space, and let $(g_i\from Y\to X_i)_{i\in I}$ be a compatible family of uniformly continuous maps. Here {\it compatible\/} means that $\phi_{ij}g_j=g_i$ whenever $i\le j$. 
Denote the inverse limit by $X=\varprojlim X_i$ and denote the canonical maps by $\phi_i\from X\to X_i$. 
Then there is a unique uniformly continuous map $g\from Y\to X$ such that $\phi_i g=g_i$ for all $i\in I$. The map $g$ exists and is unique by the general theory of inverse limits of sets, and it is uniformly continuous by Corollary~\ref{c:Initial uniformity}.

%%%%%%%%%%%%%%%%  Sums of cofinite spaces

\subsection{Sums of cofinite spaces} \label{ss:Sums of cofinite spaces}

To begin with, let $(X_i)_{i\in I}$ be an arbitrary family of uniform spaces. The {\it uniform sum\/} of this family is the disjoint union $X=\coprod_{i\in I}X_i$ endowed with the uniformity having a fundamental system of entourages consisting of all sets of the form $\bigcup_{i\in I}V_i$, where each $V_i$ is an entourage of $X_i$. 
Note that each $X_i$, when identified with its image in $X$ under the canonical inclusion map,  is a uniform subspace of $X$. For let $i_{X_i}\from X_i\to X$ be the corresponding inclusion map. Now let $U = \bigcup_{i\in I}U_i$ be an entourage over $X$. Let $(x_i,y_i)\in U_i $. Then $(x_i,y_i)\in U$. So $U_i\subseteq (i_{X_i}\times i_{X_i})^{-1}[U]$. Hence $i_{X_i}$ is uniformly continuous. Also, $(x_i,y_i)\in U\cap ((i_{X_i}\times i_{X_i})[X_i\times X_i]) \Leftrightarrow (x_i,y_i)\in U_i$. Hence $U\cap ((i_{X_i}\times i_{X_i})[X_i\times X_i]) = U_i$. 

Conversely, the next lemma gives a criterion for when a partition of a uniform space constitutes a uniform sum decomposition.

\begin{lemma}\label{partition}
Let $X$ be a uniform space and let $(X_i)_{i\in I}$  be a family of uniform subspaces that forms a partition of $X$.  Suppose that whenever $U_i$ is an entourage of $X_i$ for each 
$i\in I$, then $\bigcup_{i\in I}U_i$ is an entourage of $X$.
Then $X=\coprod_{i\in I}X_i$, the uniform sum.
\end{lemma}
\begin{proof}
Let $U$ be an entourage over $X$. Then $U_i = U\cap (X_i\times X_i)$ is an entourage over $X_i, \forall i\in I$. Now $\bigcup_{i\in I}U_i$ is an entourage over $X$ which is contained in $U$. Thus all sets of the form $\bigcup_{i\in I}V_i$, where each $V_i$ is an entourage of $X_i$, forms a fundamental system of entourages for the uniformity over $X$. Hence $X=\coprod_{i\in I}X_i$.
\end{proof} 

As a direct consequence of the above lemma one can claim that 
\begin{corollary}\label{profinite unif sp}
For a compact, Hausdorff topological space $X$ if $(X_i)_{i\in I}$ is a family of open subspaces that forms a partition of $X$ then $X = \coprod_{i\in I}X_i$.
\end{corollary}

It should be noted that the underlying topological space of a uniform sum $X$ of uniform spaces $(X_i)_{i\in I}$ is the same as the topological sum of the underlying topological spaces of the $X_i$. 

Uniform sums satisfy the following pasting lemma for uniformly continuous maps.

\begin{lemma}\label{uniform continuous maps and sum}
Let $X$ be the uniform sum of a family $(X_i)_{i\in I}$ of uniform spaces.
If $f$ is a function from $X$ to a uniform space $Y$, then $f$ is uniformly continuous if and only if each restriction $f|_{X_i}\from X_i\to Y$ is uniformly continuous.
\end{lemma}
\begin{proof}
We already have noted that the inclusion maps $i_{X_i}\from X_i\to X$ are uniformly continuous for all $i\in I$.

Now let $f\from X\to Y$ be uniformly continuous. Then $f|_{X_i}\from X_i\to Y$ can be realized as $f\circ i_{X_i}\from X_i\to Y$ and hence is uniformly continuous for all $i\in I$.\\
Conversely, let each restriction $f|_{X_i}\from X_i\to Y$ be uniformly continuous. Then for any entourage $U$ over $Y$, the set $U_i = (f|_{X_i}\times f|_{X_i})^{-1}(U)$ is an entourage over $X_i$ for all $i\in I$. Thus $R = \bigcup_{i\in I}U_i$ is an entourage over $X$. Let $(x,y)\in R $ so that there exists $i\in I$ such that $(x,y)\in U_i$. Now $(f|_{X_i}\times  f|_{X_i})(x,y)\in U$ so that $(f\times f)(x,y)\in U$ which implies that $(x,y)\in (f\times f)^{-1}[U]$. Hence $R\subseteq (f\times f)^{-1}[U]$. Thus $f$ is uniformly continuous. 
\end{proof}

In general, the uniform sum of a family of cofinite spaces may not be a cofinite space. 
However, this is true for {\it finite\/} uniform sums:

\begin{proposition} \label{p:Sums of cofinite spaces}
The uniform sum $X$ of a finite family $(X_i)_{i=1}^n$ of cofinite spaces is a cofinite space. 
\end{proposition}

%%%%%%%%%%%%%%%  Quotients of cofinite spaces

\subsection{Quotients of cofinite spaces} \label{ss:Quotients of cofinite spaces}

In general, there is no obvious way to form quotients of uniform spaces. However, there is a nice way to do this in the special case of cofinite spaces.  
First let us recall the correspondence theorem from set theory. 

\begin{note}[Correspondence Theorem]
Let $q\from X\to Y$ be a surjective function and let 
$K=q^{-1}q=\{(x_1,x_2)\in X\times X \mid q(x_1)=q(x_2)\}$. Then there is a one-to-one correspondence between the set of all equivalence relations $R$ on $X$ such that 
$K\subseteq R$ and the set of all equivalence relations on $Y$ given by
$$
R\mapsto (q\times q)[R]=qRq^{-1}.
$$
\end{note}

\begin{definition}[Uniform quotient map]\rm
Let $X$ and $Y$ be cofinite spaces. A map $q\from X\to Y$ is called a {\it uniform quotient map\/} if $q$ is surjective and if for each equivalence relation $R$ on $Y$, $R$ is a cofinite entourage if and only if $(q\times q)^{-1}[R]$ is a cofinite entourage.
\end{definition}

Uniform quotient maps of cofinite spaces satisfy a fundamental property analogous to that of quotient maps of topological spaces.

\begin{proposition} \label{p:Uniform quotient map}
Let $q\from X\to Y$ be a uniform quotient map of cofinite spaces. Then 
\begin{enumerate}
\item $q$ is uniformly continuous; 
\item a function $f$ from $Y$ to a uniform space $Z$ is uniformly continuous if and only if $f\circ q$ is uniformly continuous.
\end{enumerate}
\end{proposition}

\begin{corollary}
If $q\from X\to Y$ is a uniform quotient map of cofinite spaces, then a function $f$ from $Y$ to a cofinite space $Z$ is a uniform quotient map if and only if $f\circ q$ is a uniform quotient map.
\end{corollary}
Now we turn to constructing uniform quotients of a cofinite space. Let $X$ be a cofinite space and let $I$ denote its filter base of cofinite entourages. 
Let an equivalence relation $K$ on $X$ be given and set $I'=\{R\in I\mid K\subseteq R\}$.
Denote the canonical map from $X$ to the set of equivalence classes $X/K$ by $q\from X\to X/K$.
 
By the correspondence theorem, the collection $J=\{(q\times q)[R]\mid R\in I'\}$ is a filter base of equivalence relations on $Y$, each having finitely many equivalence classes. We see that $J$ is a fundamental system of cofinite entourages for a uniformity on the set of equivalence classes $X/K$. 
We call this uniformity the {\it quotient uniformity of $X$ modulo $K$}.

In general, the topology induced by the quotient uniformity of $X$ modulo $K$ is not as fine as the quotient topology on $X/K$. 
For this reason, we write $X/\!/K$ for the set $X/K$ endowed with the quotient uniformity of $X$ modulo $K$ and the topology it induces, reserving the notation $X/K$ for the quotient space (with the quotient topology). 

\begin{definition}[Uniform quotient space] \rm
If $K$ is an equivalence relation on a cofinite space $X$, then $X/\!/K$ is called the {\it uniform quotient space of $X$ modulo~$K$}.
\end{definition}

\begin{lemma}
Let $X$ be a cofinite space and let $K$ be an equivalence relation on $X$. 
Then the canonical map $q\from X\to X/\!/K$ is a uniform quotient map.
\end{lemma}
\begin{proof}
It is obvious that $q$ is surjective. Now let $S$ be a cofinite entourage over $X/\!/K $ so $ (q\times q) [R]\subseteq S$ for some cofinite entourage $R$ over $X$, that contains $K$. So $ R\subseteq (q\times q)^{-1}[(q\times q)[R]]\subseteq (q\times q)^{-1}[S] $. Hence $ (q\times q)^{-1}[S]$ is a cofinite entourage over $X$. Now let $(q\times q)^{-1}[T]$ is a cofinite entourage over $X$, for some equivalence relation $T$ over $X/\!/K$. Note that $(x,y)\in K$ implies that $q(x) = q(y)$ so $(q(x),q(y))\in T$. So $(x,y)\in (q\times q)^{-1}[T]$. So $K\subseteq (q\times q)^{-1}[T]$. Then $(q\times q)[(q\times q)^{-1}[T]]  = T$ is a cofinite entourage over $X/\!/K$. Hence $q$ is a uniform quotient map.
\end{proof}   

\begin{proposition}
Let $f\from X\to Y$ be a uniform quotient map of cofinite spaces and let $K=f^{-1}f$. 
Then there is an isomorphism of uniform spaces $X/\!/K\to Y$ given by $K[x]\mapsto f(x)$.
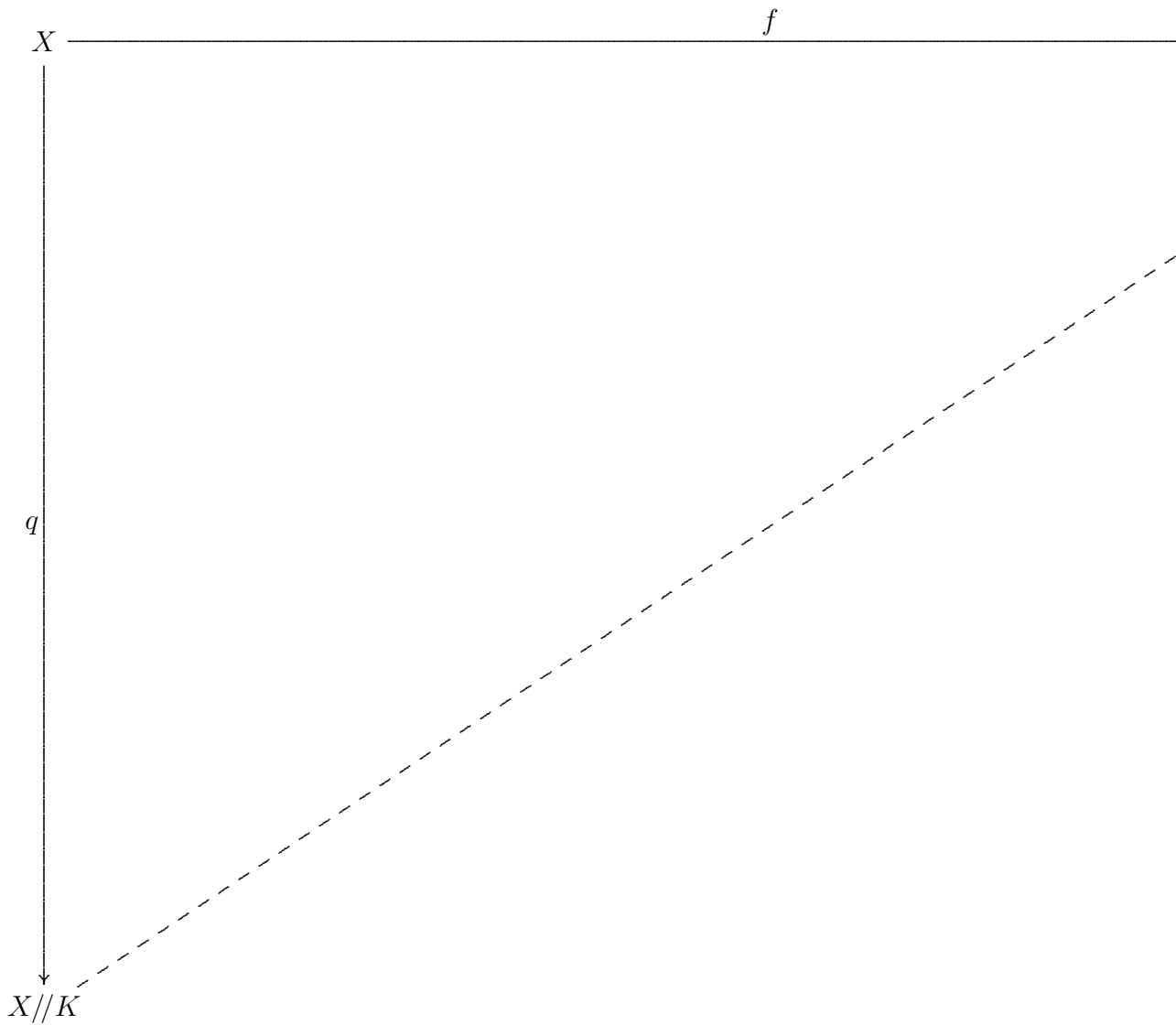
$$\begindc{\commdiag}[40]
\obj(0,0)[K]{$X/\!/K$}
\obj(15,10)[Y]{$Y$}
\obj(0,10)[X]{$X$}
\mor{X}{K}{$q$}[\atright,\solidarrow]
\mor{X}{Y}{$f$}
\mor{K}{Y}{}[\atleft,\dashArrow]
\enddc$$
\end{proposition}
\begin{proof}
Let us define $\theta\from X/\!/K\to Y$ via $\theta(K[x]) = f(x)$. Notice that $ K[x] = K[y]\Leftrightarrow (x,y)\in K\Leftrightarrow f(x) = f(y)$. Hence $\theta$ is well defined and an injection. Now let $y\in Y$. Since $f$ is a surjection, there exists $x\in X$ such that $f(x) = y$. Then $\theta(K[x]) = f(x) = y$. Thus $\theta$ is surjection as well. Now let $R$ be a cofinite entourage over $Y$. Then $(f\times f)^{-1}[R]$ is an cofinite entourage over $X$ containing $K$.Thus we claim $(q\times q)[(f\times f)^{-1}[R]]$ is a cofinite entourage over $X/\!/K$. Let $(K[x],K[y])\in(q\times q)[(f\times f)^{-1}[R]]$. Then we get $(p,r)$ in $(f\times f)^{-1}[R]$ such that $K[x]=K[p]$ and $K[y]=K[r]$ which implies that $(\theta(K[x]),\theta(K[y])) = (\theta(K[p]),\theta(K[r])) = (f(p),f(r))\in R$. This shows that $(\theta\times \theta)[(q\times q)[(f\times f)^{-1}[R]]]\subseteq R$ and thus $(q\times q)[(f\times f)^{-1}[R]]$ is a subset of $(\theta\times \theta)^{-1}[R]$. Hence $\theta$ is uniformly continuous.

Now let $S$ be a cofinite entourage over $X/\!/K$. Then there exists $T$ a cofinite entourage over $X$, containing $K$ such that $(q\times q)[T]\subseteq S$. But then $(f\times f)[T]$ is a cofinite entourage over $Y$. Moreover we have $(f\times f)[T] = (\theta\times\theta)[(q\times q)[T]]\subseteq (\theta\times\theta)[S] = (\theta^{-1}\times\theta^{-1})[S]$. Hence $\theta^{-1}$ is uniformly continuous as well.  Thus our claim follows. 
\end{proof}

It should be noted that, although a uniform quotient space $X/\!/K$ has a fundamental system of entourages consisting of cofinite entourages, it may not be Hausdorff, even if $X$ is a cofinite Hausdorff space. We give the following answer to the question as to when $X/\!/K$ is a Hausdorff cofinite space. 

\begin{proposition} \label{p:Hausdorff quotient}
Let $X$ be a cofinite space and let $I$ be the filter base of cofinite entourages of $X$. If $K$ is any equivalence relation on $X$, then the following conditions are equivalent:
\begin{enumerate}
\item $X/\!/K$ is a Hausdorff cofinite space;
\item $\bigcap\{R\mid 
\hbox{$R\in I$ and $K\subseteq R$}\}=K$. 
\end{enumerate}
\end{proposition}
\begin{proof}
(1) $\Rightarrow$ (2): 

Let $X/\!/K$ be Hausdorff. Since $K\subseteq R$ for all $R$ in $I$, we obtain $K\subseteq \bigcap \{R\mid R\in I$ and $K\subseteq R\}$. Now let $(x,y)\in \bigcap \{R\mid R\in I$ and $K\subseteq R\}$. This implies $(q(x),q(y))\in (q\times q)[R]$, for all $R\in I$ whenever $R$ contains $K$. But $X/\!/K$ is Hausdorff so we conclude that $q(x)=q(y)$. Thus $K[x]=K[y]$. Hence $(x,y)\in K$. So, 
$$
\bigcap \{R\mid R\in I, K\subseteq R\} = K
$$

(2) $\Rightarrow$(1):

Let us now take $\bigcap \{R\mid R\in I$ and $K\subseteq R\} = K$. Now if $K[x]\neq K[y]$ in $X/\!/K$, we have $(x,y)\notin K$. Hence there exists some $R\in I$ containing $K$ such that $(x,y)\notin R$. But then $(q(x),q(y)) = (K[x],K[y])$ does not belong to $(q\times q)[R]$. Otherwise $\exists(t,s)\in R$ so that $q(t) = q(x)$ and $q(y) = q(s)$. Then $(x,t)\in K\subseteq R, (t,s)\in R, (s,y)\in K\subseteq R$, which implies $(x,y)\in R$, a contradiction. Hence $X/\!/K$ is a Hausdorff cofinite space.
\end{proof}
Note that we do not even require $X$ to be Hausdorff in the above cases.

In some important special cases, the uniform quotient space of a cofinite space $X$ modulo an equivalence relation $K$ is equal to its quotient space $X/K$ (as topological spaces). To give a necessary and sufficient condition for this to hold, we first make some general observations about quotients of topological spaces. 

Let $K$ be an equivalence relation on a topological space $X$ and denote the canonical quotient map by $q\from X\to X/K$. 
We say that a subset $B\subseteq X$ is {\it $K$-saturated\/} if $K[B]=B$. 
It is easy to check that the intersection of any family of $K$-saturated subsets is again $K$-saturated.

Let $\{B_{\lambda}\mid\lambda\in\Lambda\}$ be a family of $K$-saturated subsets of $X$. Then for all $\lambda$ in $\Lambda, K[B_{\lambda}] = B_{\lambda}$. Let $B = \bigcap_{\lambda\in\Lambda}B_{\lambda}$ and $x\in K[B]$. Then there exists some $b$ in $B$ such that $x\in K[b]$. Hence $x\in K[b]\subseteq K[B_{\lambda}] = B_{\lambda} $, for all $\lambda\in\Lambda$. Thus $x\in\bigcap_{\lambda\in\Lambda}B_{\lambda} = B$. So $K[B]\subseteq B\subseteq K[B]$. Thus $K[B] = B$.
     
Hence, for any subset $A$ of $X$, there is a unique smallest $K$-saturated closed subset ${\overline A}^s$ of $X$ with 
$A\subseteq{\overline A}^s$; simply let ${\overline A}^s$ be the intersection of the family of all closed $K$-saturated subsets of $X$ that contain $A$.

\begin{lemma}\label{p:q(A)}
For any subset $A$ of $X$, we have $q({\overline A}^s)=\overline{q (A)}$.
\end{lemma}
\begin{proof}
Let us first see that $q^{-1}(q(\overline{A}^s)) = K[\overline{A}^s] = \overline A^s$.

Now $x\in q^{-1}(q(\overline{A}^s)) \Leftrightarrow q(x)\in q(\overline{A}^s) \Leftrightarrow$ there exists $t\in\overline{A}^s$ such that $q(t) = q(x) \Leftrightarrow (t,x)\in K \Leftrightarrow x\in K[t]\subseteq K[\overline{A}^s]$. So $q^{-1}(q(\overline{A}^s)) = \overline{A}^s$ and hence $q(\overline{A}^s)$ is closed in $X/K$. Then $A\subseteq\overline{A}^s$ implies $q(A)\subseteq q(\overline{A}^s)$ and thus $\overline{q(A)}\subseteq\overline{q(\overline{A}^s)} = q(\overline{A}^s)$.

Since $\overline{q(A)}$ is closed in $X/K, q^{-1}(\overline{q(A)})$ is closed in $X$. Clearly, $A\subseteq\overline{A}\subseteq q^{-1}(q(\overline{A}))\subseteq q^{-1}(\overline{q(A)})$. Now let $x\in K[q^{-1}(\overline{q(A)})]$. This implies that there exists $t\in q^{-1}(\overline{q(A)})$ such that $x\in K[t]$. Then $(t,x)\in K$, where $q(t)\in \overline{q(A)}$, so $ q(x) = q(t)\in\overline{q(A)} $ and thus $ x\in q^{-1}(\overline{q(A)})$. Hence $K[q^{-1}(\overline{q(A)})]\subseteq q^{-1}(\overline{q(A)}) \subseteq K[q^{-1}(\overline{q(A)})]$. So $K[q^{-1}(\overline{q(A)})] = q^{-1}(\overline{q(A)})$. Thus $q^{-1}(\overline{q(A)})$ is a $K$- saturated closed subset of $X$ containing $A$ and hence $\overline{A}^s\subseteq q^{-1}(\overline{q(A)})$. Hence we get $q(\overline{A}^s)\subseteq q(q^{-1}(\overline{q(A)})) = \overline{q(A)}$. Hence our claim $q(\overline{A}^s) = \overline{q(A)}$.       
\end{proof}   
 
\begin{theorem}\label{p:iff}
Let $X$ be a cofinite space and let $K$ be an equivalence relation on $X$. 
\begin{enumerate}
\item The identity map $\id\from X/K\to X/\!/K$ is a continuous bijection.
\item The identity map $\id\from X/K\to X/\!/K$ is a homeomorphism $($i.e., the topology induced by the quotient uniformity of $X$ modulo $K$ and the quotient topology are the same$)$ if and only if $K$ satisfies the property: for each subset $A\subseteq X$, 
the $K$-saturated closure 
${\overline A}^s=\bigcap R[A]$, as $R$ runs through all cofinite entourages of $X$ such that $K\subseteq R$.
\end{enumerate}
\end{theorem}
\begin{proof} We will prove the results in the order they appear.
\begin{enumerate}
\item Its obvious that $\id\from X/K\to X/\!/K$ is  a bijection. Now let $O$ be open in $X/\!/K$. More over let us take $x\in q^{-1}(id^{-1}(O))$ so that $ K[x]\in \id^{-1} (O) = O$. Hence there is a cofinite entourage $R$ over $X$ such that  $K[x]\in (q\times q)[R][K[x]]\subseteq O$. Now let $t\in R[x]$. Hence $(x,t)\in R$ which implies that $(K[x],K[t])\in (q\times q)[R]$. Therefore $K[t]\in (q\times q)[R][K[x]]\subseteq O$, so $t\in q^{-1}(O)$. Hence $x\in R[x]\subseteq q^{-1}(O)$. Hence $q^{-1}(O)$ is open in $X$. Thus $O$ is open in $X/K$, proving the continuity of $\id$.

\item Let us first assume that $\id$ is a homeomorphism between $X/K$ and $X/\!/K$. Let $I^{\prime} = \{R\mid R$ is a cofinite entourage over $X$ and $K\subseteq R\}$. Now for any subset $Q$ of $X/\!/K$ we observe that the closure of $Q$, $\overline{Q} = \bigcap_{R\in I^{\prime}}(q\times q)[R][Q]$. As $\id$ is a homeomorphism it is also a closed map. So $q(A) = A_K$. We also can now claim that $\bigcap_{R\in I^{\prime}}(q\times q)[R][A_K] = A_{K_q}$ and so it follows from Lemma~\ref{p:q(A)} 

$$
\overline{A}^s=q^{-1}(\overline{A_K}^{X/K})=q^{-1}(\overline{A_K}^{X/\!/K})=q^{-1}(A_{K_q})
$$ 

If $x\in \overline{A}^s$ it follows that $q(x)\in A_{K_q}$ which implies that for all $R\in I^{\prime}$ there exists $a_R\in A$ such that $q(x)\in (q\times q)[R][q(a_R)]$. Then $(q(a_R),q(x))\in (q\times q)[R]$. Hence $\exists(t_{1_R},t_{2_R})\in R$, for all $R\in I^{\prime}$ such that $q(t_{1_R}) = q(a_R)$ and $q(t_{2_R}) = q(x)$. So $(a_R,t_{1_R})\in K\subseteq R, (t_{1_R},t_{2_R})\in R, (t_{2_R},x)\in K\subseteq R, \forall R\in I^{\prime}$. Thus $(a_R,x)\in R, \forall R\in I^{\prime}$. Sot $x\in R[a_R], \forall R\in I^{\prime}$ and thus $x\in\bigcap_{R\in I^{\prime}}R[A]$. 

On the other hand, let us take $y\in \bigcap_{R\in I^{\prime}}R[A]$. This implies that for all $R$ in $I^{\prime}$ there exists $b_R\in A$ such that $y\in R[b_R]$. So $(b_R,y)\in R$, for all $R\in I^{\prime}$. This implies that $(q(b_R),q(y))$ is in $(q\times q)[R]$, for all $R\in I^{\prime}$, so $q(y)\in A_{K_q}$ and therefore $y\in q^{-1}(A_{K_q}) = \overline{A}^s$. Thus $\overline{A}^s = \bigcap_{R\in I^{\prime}}R[A]$. Let us now note that $A_{K_q} = \bigcap(qRq^{-1})[A_K] = \bigcap (qRK)[A] = \bigcap q(R[A])$ as $K\subseteq R$ and for all $R\in I^{\prime}$, Hence $A_{K_q} = (qq^{-1})(\bigcap q(R[A])) = q(\bigcap(KR)[A]) = q(\bigcap_{R\in I^{\prime}}R[A])$.

Conversely, let us assume that $\overline{A}^s = \bigcap_{R\in I^{\prime}}R[A]$. 
We will first see that for any subset $A$ of $X/K$, $q^{-1}(A)$ is $K$-saturated. For, $x\in K[q^{-1}(A)]$ implies that there exists $a\in q^{-1}(A)$, such that $(a,x)\in K$ and then $q(x) = q(a)\in A$. Hence $x\in q^{-1}(A)$. So $K[q^{-1}(A)]\subseteq q^{-1}(A)\subseteq K[q^{-1}(A)]$. Hence $K[q^{-1}(A)] = q^{-1}(A)$. Now let $B$ be closed in $X/K$. Then $C = q^{-1}(B)$ is closed in $X$.  Hence $C$ is a closed $K$-saturated subset of $X$. Hence we claim that $C = \overline{C}^{s} = \bigcap_{R\in I^{\prime}}R[C]$.\\

We now want to prove that $B = q(C) = \bigcap_{R\in I^{\prime}}(q\times q)[R][q(C)]$. To see this let $s\in q(C)$. This implies that there exists $t\in C$ such that $s = q(t)$ and for some $b_R\in R[C], \forall R\in I^{\prime}$ and $(b_R,t) \in R$. Then $(q(b_R),q(t))\in (q\times q)[R]$. Hence, for all $R\in I^{\prime}$ there exists $s\in (q\times q)[R][q(b_R)] \subseteq  (q\times q)[R][q(C)]$ and thus $s\in \bigcap_{R\in I^{\prime}}(q\times q)[R][q(C)]$. For the other way, let $z\in \bigcap_{R\in I^{\prime}}(q\times q)[R][q(C)]$. This implies there exists $c_R\in C$ such that $(q(c_R),z)\in (q\times q)[R]$, for all $R\in I^{\prime}$ and so for all $R\in I^{\prime}$, there exists $(m,n)$ in $R$ such that $q(m) = q(c_R), q(n) = z$. Then, for all $R\in I^{\prime}, (c_R,m)\in K\subseteq R$ and so $(c_R,n)\in R$, for all $R\in I^{\prime}$. Consequently, for all $R\in I^{\prime}, n$ in $R[c_R]\subseteq R[C]$, so $n\in \bigcap_{R\in I^{\prime}}R[C] = C$. Thus $z = q(n)\in q(C)$. S we get our final claim $
B = \bigcap_{R\in I^{\prime}}(q\times q)[R][q(C)] = \overline B^{X/\!/K}
$ and so $\id$ is a closed map and thus is a homeomorphism.                         
\end{enumerate} 
\end{proof}
\begin{corollary}
If $K$ is an equivalence relation on a cofinite space $X$ such that $X/K$ is compact and 
$\bigcap\{R\mid \hbox{$R\in I$ and $K\subseteq R$}\}=K$, then 
$\id\from X/K\to X/\!/K$ is a homeomorphism.
\end{corollary}
\begin{proof}
By Proposition~\ref{p:Hausdorff quotient}, $X/\!/K$ is Hausdorff and so $\id$ is a continuous bijection from a compact space to a Hausdorff space and thus is a homeomorphism.
\end{proof}
\begin{corollary}
If $X$ is a cofinite space and $R$ is a cofinite entourage of $X$, then 
$\id\from X/R\to X/\!/R$ is a homeomorphism.
\end{corollary}
\begin{proof}
First let us take $I = \{S\mid S$ is a cofinite entourage over $X$ and $R\subseteq S\}$.
Since $R$ is a cofinite entourage, $X/R$ is finite discrete and thus compact. Also 
$$
\bigcap\{S\in I; R\subseteq S\}=R
$$ 
and thus by Proposition~\ref{p:Hausdorff quotient}, $X/\!/R$ is Hausdorff, so by the  previous corollary, $\id\from X/R\to X/\!/R$ is a homeomorphism.  
\end{proof} 
%%%%%%%%%%Section: Inverse Limits of Compact Hausdorff Spaces     %%%%%%%%%%   
\section{Inverse limits of compact Hausdorff spaces}
We begin with some observations about general inverse systems of topological spaces. Let $(X_i,\phi_{ij})$ be an inverse system of topological spaces indexed by a directed set~$I$. 

\begin{note} \label{n:projlim1}
Let $X$ denote the inverse limit of $(X_i,\phi_{ij})$ and let $\phi_i\from X\to X_i$ be the canonical map for each $i\in I$. 
\begin{enumerate}
\item The family of sets $\phi_i^{-1}(U_i)$, where $i\in I$ and $U_i$ is open in $X_i$, is a basis for the topology of $X$.
\item Let $A$ be a subset of $X$ and write $A_i=\phi_i(A)$ for each $i\in I$. Then
$$
\overline A=\bigcap_{i\in I}\phi_i^{-1}(\overline{A_i}) =\varprojlim\overline{A_i}.
$$
\item If $A$ is a subset of $X$ satisfying $\phi_i(A)=X_i$ for all $i\in I$, then $A$ is dense in $X$.
\item If $f\from Y\to X$ is a function from a space $Y$, then $f$ is continuous if and only if each composition $\phi_i f$ is continuous.
\end{enumerate}
\end{note}

Next we specialize to compact Hausdorff spaces.

\begin{note}
Let $(X_i,\phi_{ij})$ be an inverse system of non-empty compact Hausdorff spaces indexed by a directed set $I$. Then the inverse limit $X=\varprojlim X_i$ has the following properties:
\begin{enumerate}
\item $X$ is a non-empty compact Hausdorff space.
\item $\phi_i(X)=\bigcap_{j\ge i}\phi_{ij}(X_j)$ for each $i\in I$.
\item If $A$, $B$ are disjoint closed subsets of $X$, then there exists $i\in I$ such that $\phi_i(A)$, $\phi_i(B)$ are disjoint closed subsets of $X_i$.
\item If $Y$ is a discrete space and $f\from X\to Y$ is a continuous map, then $f$ factors through some $X_k$; i.e., for some $k\in I$ there is a continuous map $h\from X_k\to Y$ such that $f=h\phi_k$.
\end{enumerate}
\end{note}

\note \label{n:ProSpace}
The following conditions are equivalent for any compact Hausdorff space $X$:
\begin{enumerate}
\item X is totally disconnected;
\item the clopen subsets of $X$ form a basis for its topology;
\item $\bigcap\{R \mid\hbox{$R$ is a co-discrete equivalence relation on $X$}\}$ is equal to the diagonal of $X\times X$;
\item $X$ is Hausdorff cofinite space, when endowed with the unique uniform structure;
\item $X$ is the inverse limit of an inverse system $(X_i,\phi_{ij})$ of finite discrete spaces.
\end{enumerate}
\endnote

\begin{lemma}
Let $X$ be a compact Hausdorff space and let $x\in X$. Then the intersection of all clopen subsets of $X$ that contain~$x$ is equal to the component of $x$.
\end{lemma}

\begin{definition}[Profinite space]\rm
A compact Hausdorff space $X$ that satisfies the equivalent conditions of the previous result is called a {\it profinite space}.
\end{definition}

We will always assume that a profinite space $X$ is endowed with the unique uniform structure that induces its topology, and hence, by~\ref{n:ProSpace}(4), $X$ is a Hausdorff cofinite space. 
Thus profinite spaces are precisely the compact, Hausdorff cofinite spaces.

%%%%%%%%%% Section: Topological Graphs %%%%%%%%%%

\section{Topological graphs}

A {\it topological graph\/} is a topological space $\Gamma$ that is partitioned into two closed subsets $V(\Gamma)$ and $E(\Gamma)$ together with two continuous functions $s,t\from E(\Gamma)\to V(\Gamma)$ and a continuous function $\overline{\phantom e}\from E(\Gamma)\to E(\Gamma)$ satisfying the following properties: for every $e\in E(\Gamma)$,
\begin{enumerate}
\item $\overline{e}\ne e$ and $\overline{\overline e}=e$;
\item $t(\overline e)=s(e)$ and $s(\overline e)=t(e)$.
\end{enumerate}
The elements of $V(\Gamma)$ are called {\it vertices\/}. An element $e\in E(\Gamma)$ is called a ({\it directed}) {\it edge\/} with {\it source\/} $s(e)$ and {\it target\/} $t(e)$; the edge $\overline e$ is called the {\it reverse\/} or {\it inverse\/} of $e$. 

A {\it map of graphs\/} $f\from\Gamma\to\Delta$ is a function that maps vertices to vertices, edges to edges, and preserves sources, targets, and inverses of edges. Analogously, we will call a map of graphs a {\it graph isomorphism\/} if and only if it is a bijection. 

An {\it orientation} of a topological graph $\Gamma$ is a closed subset $E^+(\Gamma)$ consisting of exactly one edge in each pair $\{e,\overline e\}$. 
In this situation, setting 
$E^-(\Gamma)=\{e\in E(\Gamma) \mid \overline e\in E^+(\Gamma)$\} we see that $E(\Gamma)$ is a disjoint union of the two closed (hence also open) subsets  
$E^+(\Gamma)$, $E^-(\Gamma)$.   

\begin{note}
Let $\Gamma$ be a topological graph. The following are equivalent:
\begin{enumerate}
\item $\Gamma$ admits an orientation;
\item there exists a continuous map of graphs from $\Gamma$ to the discrete graph with a single vertex and a single edge and its inverse;
\item there exists a continuous map of graphs $f\from\Gamma\to\Delta$ for some discrete graph $\Delta$.
\end{enumerate}
 \end{note}

Conceivably there are topological graphs that do not admit closed orientations. However such graphs will not concern us. 
Therefore, unless otherwise stated, by a topological graph we will henceforth mean a topological graph that admits an orientation. 

We will be interested in equivalence relations on graphs that are compatible with the graph structure:

\begin{definition}[Compatible equivalence relation]\rm
An equivalence relation $R$ on a graph $\Gamma$ is {\it compatible\/} if the following properties hold:
\begin{enumerate}
\item $R=R_V\cup R_E$ where $R_V$, $R_E$ are equivalence relations on $V(\Gamma)$, $E(\Gamma)$, precisely the restriction of $R$;
\item if $(e_1,e_2)\in R$, then $(s(e_1),s(e_2))\in R$, $(t(e_1),t(e_2))\in R$, and 
$(\overline e_1,\overline e_2)\in R$;
\item for all $e\in E(\Gamma)$, $(e,\overline e)\notin R$;
\end{enumerate}
\end{definition}

\begin{note}
If $K$ is a compatible equivalence relation on $\Gamma$, then there is a unique way to make $\Gamma/K$ into a graph such that the canonical map $\Gamma\to\Gamma/K$ is a map of graphs. It is defined by setting $s(K[e])=K[s(e)]$, $t(K[e])=K[t(e)]$, and 
$\overline{K[e]}=K[\overline e]$.\\
 
Conversely, if $\Delta$ is a graph and $f\from\Gamma\to\Delta$ is a surjective map of graphs, then
$K=f^{-1}f=\{(a,b)\in\Gamma\times\Gamma\mid f(a)=f(b)\}$ is a compatible equivalence relation on $\Gamma$ and $f$ induces an isomorphism of graphs such that $\Gamma/K\cong\Delta$.
\end{note}
 
\begin{note}
If $R_1$ and $R_2$ are compatible equivalences on $\Gamma$, then so is $R_1\cap R_2$.
\end{note}
 
\begin{theorem} \label{t:cofinal}
Let $R$ be any cofinite equivalence relation on a topological graph~$\Gamma$. Then there exists a compatible cofinite equivalence relation $S$ on $\Gamma$ such that $S\subseteq R$.
\end{theorem}

\begin{proof}
Extend the source and target maps $s,t\from E(\Gamma)\to V(\Gamma)$ to all of $\Gamma$ so that they are both the identity map on~$V(\Gamma)$. Then $s,t\from\Gamma\to\Gamma$ are continuous maps satisfying the following properties:
\begin{itemize}
\item $s^2=s$, $t^2=t$, $st=t$, and $ts=s$;
\item $s(x)=x \iff t(x)=x \iff x\in V(\Gamma)$.
\end{itemize}

Similarly, extend the edge inversion map 
$\overline{\phantom e}\from E(\Gamma)\to E(\Gamma)$ to all of  $\Gamma$ by also letting it be the identity map on $V(\Gamma)$. Then $\overline{\phantom e}\from\Gamma\to\Gamma$ is a continuous map satisfying the following conditions for all $x\in\Gamma$:
\begin{itemize}
\item $\overline{\overline x}=x$;
\item $\overline x=x \iff x\in V(\Gamma)$;
\item $s(\overline x)=t(x)$ and $t(\overline x)=s(x)$. 
\end{itemize}

Now define
$S_1=\{ (x,y)\in\Gamma\times\Gamma\mid(s(x),s(y))\in R \}=(s\times s)^{-1}[R]$,  $S_2=\{ (x,y)\in\Gamma\times\Gamma\mid(t(x),t(y))\in R \}=(t\times t)^{-1}[R]$, and
$S_3=\{ (x,y)\in\Gamma\times\Gamma\mid(\overline x,\overline y)\in R)
=(\overline{\phantom e}\times\overline{\phantom e})^{-1}[R]$. Then, by Theorem~\ref{correspondence}, $S_1$, $S_2$, $S_3$ are cofinite equivalence relations on $\Gamma$. 
Let $S_4=R\cap S_1\cap S_2\cap S_3$ and observe that
\begin{enumerate}
\item[(i)] $S_4$ is a cofinite equivalence relation on $\Gamma$;
\item[(ii)] if $(e_1,e_2)\in S_4$, then $(s(e_1),s(e_2))\in S_4$, $(t(e_1),t(e_2))\in S_4$, and $(\overline e_1,\overline e_2)\in S_4$.
\end{enumerate}

Finally, choose a closed orientation $E^+(\Gamma)$ of $\Gamma$ and form the restrictions 
$S_V=S_4\cap[V(\Gamma)\times V(\Gamma)]$, 
$S_{E^+}=S_4\cap[E^+(\Gamma)\times E^+(\Gamma)]$, and
$S_{E^-}=S_4\cap[E^-(\Gamma)\times E^-(\Gamma)]$.
Then it is easy to check that $S=S_V\cup S_{E^+}\cup S_{E^-}$ is a compatible cofinite equivalence relation on $\Gamma$ and $S\subseteq R$, as required.
\end{proof}

The previous proof actually shows a little more, which is worth noting. Given a closed orientation $E^+(\Gamma)$ for $\Gamma$, we say that a compatible equivalence relation $R$ on $\Gamma$ is {\it orientation preserving\/} if whenever $(e,e')\in R$ and $e\in E^+(\Gamma)$, then also $e'\in E^+(\Gamma)$. 
Since the equivalence relation $S$ that we constructed in the proof of Theorem~\ref{t:cofinal} is also orientation preserving, we proved the following stronger result.  

\begin{corollary} \label{c:orientation preserving}
Let $\Gamma$ be a topological graph with a specified closed orientation $E^+(\Gamma)$. Then for any cofinite equivalence relation $R$ on $\Gamma$, there exists a compatible orientation preserving cofinite equivalence relation $S$ on $\Gamma$ such that 
$S\subseteq R$.
\end{corollary}

\begin{corollary} \label{c:ProGraph}
If $\Gamma$ is a compact Hausdorff totally disconnected topological graph, then its compatible cofinite equivalence relations form a fundamental system of entourages for the unique uniform structure that induces the topology of~$\Gamma$. 
\end{corollary}

%%%%%%%%%%%%%%  Profinite graphs

\begin{definition}[Profinite graph]\rm
A compact Hausdorff totally disconnected topological graph $\Gamma$ is called a {\it profinite graph}.
\end{definition}

As for any compact Hausdorff space, we will view a profinite graph as a uniform space endowed with the unique uniformity that induces its topology. Thus, Corollary~\ref{c:ProGraph} states that the collection of all compatible cofinite equivalence relations on a profinite graph $\Gamma$ form a fundamental system of entourages.

%%%%%%%%%% Section: Cofinite Graphs %%%%%%%%%%

\section{ Cofinite graphs} \label{s:cofinite}

By a {\it uniform topological graph\/} we mean a topological graph $\Gamma$ endowed with a uniform structure that induces its topology such that $\Gamma$ is the uniform sum of its uniform subspaces $V(\Gamma)$, $E(\Gamma)$ and  the maps 
$s,t\from E(\Gamma)\to V(\Gamma)$ and 
$\overline{\phantom e}\from E(\Gamma)\to E(\Gamma)$
are uniformly continuous.

\begin{note}
If $f\from\Gamma\to\Delta$ is a uniformly continuous map of uniform topological graphs then for any compatible cofinite equivalence relation $R$ over $\Delta, (f\times f)^{-1}(R)$ is a compatible cofinite equivalence relation over $\Gamma$.
\end{note}  

We will concentrate our attention on uniform topological graphs of the following type.

\begin{definition}[Cofinite graph] \label{d:Cofinite graph} \rm
A {\it cofinite graph\/} is an abstract graph $\Gamma$ endowed with a Hausdorff uniformity such that the compatible cofinite entourages of $\Gamma$ form a fundamental system of entourages (i.e. every entourage of $\Gamma$ contains a compatible cofinite entourage). 
\end{definition}
\begin{lemma}
Let $\Gamma$ be a cofinite graph. Then $\Gamma$ is a uniform topological graph. In particular,
\begin{enumerate} 
\item[1.] $V(\Gamma)$ and $E(\Gamma)$ are clopen subsets of $\Gamma$;
\item[2.] $\Gamma$ is the uniform sum of its uniform subspaces $V(\Gamma), E(\Gamma)$;
\item[3.] $s, t\from E(\Gamma)\to V(\Gamma)$ and $\overline{\phantom e}\from E(\Gamma)\to E(\Gamma)$ are uniformly continuous maps.
\end{enumerate}
\end{lemma}

\begin{lemma}
Profinite graphs are precisely the compact cofinite graphs.
\end{lemma}

Let $\Gamma$ be a cofinite graph and let $I$ be a fundamental system of compatible cofinite entourages of $\Gamma$. Then we see  by Note ~\ref{n:Cofinite Properties} that
\begin{enumerate}
\item[(i)] $\bigcap_{R\in I}R=D(\Gamma)$, the diagonal in $\Gamma\times\Gamma$; 
\item[(ii)] $\Gamma$ is totally disconnected; 

\item[(iii)] if $A$ is any subset of $\Gamma$, then 
$\overline A=\bigcap_{R\in I}R[A]$\\

\end{enumerate}
The following lemma is an immediate consequence of Proposition~\ref{uniform continuous maps and sum}.

\begin{lemma} \label{l:pasting lemma}
Let $\Gamma$ be a cofinite graph and let $Z$ be a cofinite space. Then a map $f\from\Gamma\to Z$ is uniformly continuous if and only if both the restrictions 
$f|_{V(\Gamma)}$ and $f|_{E(\Gamma)}$ are uniformly continuous.
\end{lemma}

As one application of this lemma, we can extend the source map $s\from E(\Gamma)\to V(\Gamma)$ of a cofinite graph $\Gamma$ to a map $s\from\Gamma\to\Gamma$ by letting it be the identity map on $V(\Gamma)$. By Lemma~\ref{l:pasting lemma}, the extension $s\from\Gamma\to\Gamma$ is also uniformly continuous. We can similarly extend the target and inversion maps. Thus, when it is convenient to do so, we may assume that the source, target, and inversion maps are uniformly continuous maps 
$s,t,\overline{\phantom e}\from\Gamma\to\Gamma$ whose fixed points are precisely the vertices of $\Gamma$. 

%%%%%%%%%%%%%%  {Cofinite subgraphs}

\subsection{Uniform subgraphs}

Let $\Gamma$ be a cofinite graph. A subgraph $\Sigma$ endowed with the uniformity induced on it by $\Gamma$ is  called a {\it uniform subgraph\/} of $\Gamma$.

Let us observe that the subgraph $\Sigma$ of a cofinite graph $\Gamma$ is itself a cofinite graph, as because if $R$ is a compatible cofinite entourage over $\Gamma$ then so is $R\cap(\Sigma\times \Sigma)$ over $\Sigma$.

%%%%%%%%%%%%%% Inverse Limits of Cofinite Graphs

\subsection{Inverse Limits of Cofinite Graphs}

Turning to inverse limits of cofinite graphs, let $(\Gamma_i,\phi_{ij})$ be an inverse system of sets indexed by a directed set $I$. We say that $(\Gamma_i,\phi_{ij})$ is an {\it  inverse system of cofinite graphs\/} if (i)~each $\Gamma_i$ is a cofinite graph, 
and (ii)~for all $i\le j$, $\phi_{ij}\from\Gamma_j\to\Gamma_i$ is a uniformly continuous map of graphs. 

As in Section~\ref{ss:Inverse limits of cofinite spaces}, we endow the set $\Gamma=\varprojlim\Gamma_i$ with the coarsest uniformity such that the canonical maps $\phi_i\from\Gamma\to\Gamma_i$ are uniformly continuous. Then by Proposition~\ref{p:Inverse limit of cofinite spaces}, $\Gamma$ is a cofinite space. 
Furthermore, we make the following observation.

\begin{lemma} \label{l:Unique graph structure}
The set $\Gamma$ admits a unique graph structure such that the maps $\phi_i\from\Gamma\to\Gamma_i$ are maps of graphs. 
\end{lemma}

\begin{proof}
First of all, we claim that for all $i,j\in I$, 
$$
\phi_i^{-1}[V(\Gamma_i)]=\phi_j^{-1}[V(\Gamma_j)].
$$ 
To see this, choose $k\in I$ such that $k\ge i$ and $k\ge j$. Then $\phi_i=\phi_{ik}\phi_k$ and $\phi_j=\phi_{jk}\phi_k$. 
So 
$
\phi_i^{-1}[V(\Gamma_i)]=\phi_k^{-1}[\phi_{ik}^{-1}[V(\Gamma_i)]]
=\phi_k^{-1}[V(\Gamma_k)]
$
as $\phi_{ik}\from\Gamma_k\to\Gamma_i$ is a map of graphs. Similarly, 
$\phi_j^{-1}[V(\Gamma_j)]=\phi_k^{-1}[V(\Gamma_k)]$ and the claim follows. 
Now it also follows that for all $i,j\in I$, 
$$
\phi_i^{-1}[E(\Gamma_i)]=\phi_j^{-1}[E(\Gamma_j)].
$$

For the desired graph structure on $\Gamma$, the vertex and edge sets must be the subsets satisfying:
$$
V(\Gamma)=\phi_i^{-1}[V(\Gamma_i)]\quad\hbox{and}\quad
E(\Gamma)=\phi_i^{-1}[E(\Gamma_i)]
$$
for all $i\in I$.

It remains to see that there is a unique way to define the source, target, and inversion maps so that the $\phi_i$ are maps of graphs. 
We begin by extending the source, target, and inversion maps to functions 
$s,t,\overline{\phantom e}\from\Gamma_i\to\Gamma_i$, whose fixed points are precisely the vertices of $\Gamma_i$, for each $i\in I$.
Then for each $i\in I$, let $s_i=s\phi_i\from\Gamma\to\Gamma_i$.
Note that for $i\le j$, 
$$
\phi_{ij}s_j=\phi_{ij}s\phi_j=s\phi_{ij}\phi_j=s\phi_i=s_i.
$$
So the family of functions $(s_i\from\Gamma\to\Gamma_i)_{i\in I}$ determine a unique function $s\from\Gamma\to\Gamma$ such that $\phi_is=s_i=s\phi_i$ for all $i\in I$.
Similarly, there exist unique functions  
$t,\overline{\phantom e}\from\Gamma\to\Gamma$ such that all $\phi_it=t\phi_i$ and 
$\phi_i\overline{\phantom e}=\overline{\phantom e}\phi_i$. Let us now check that with these maps, $\Gamma$ is a graph. If possible, let $x\in V(\Gamma)\cap E(\Gamma)$. Hence $\forall i\in I, \phi_i(x)\in V(\Gamma_i)\cap E(\Gamma_i)$, which is a contradiction. Also $\forall x\in \Gamma$ and for all $i\in I, \phi_i(x)\in \Gamma_i = V(\Gamma_i)\cup E(\Gamma_i)$. Hence $x\in\phi_i^{-1}(V(\Gamma_i))\cup \phi_i^{-1}(E(\Gamma_i)) = V(\Gamma)\cup E(\Gamma)\subseteq \Gamma$. Thus we get $\Gamma = V(\Gamma)\cup E(\Gamma)$. Now, if possible, let there exist $e\in E(\Gamma)$ such that $e = \overline e$. Hence for all $i$ in $I, \phi_i(e) = \phi_i(\overline e)$ and thus $\phi_i(e) = \overline{\phi_i(e)}$ in $E(\Gamma_i)$ for all $i$ in $I$, a contradiction. Finally, $\phi_i(s(\overline e))=s(\phi_i(\overline e))=s(\overline{\phi_i(e)}) =t(\phi_i(e))=\phi_i(t(e))$ and $\phi_i(t(\overline e))=t(\phi_i(\overline e))=t(\overline{\phi_i(e)}) =s(\phi_i(e))=\phi_i(s(e))$ for all $i$ in $I$. Hence it follows that $s(\overline e) = t(e), t(\overline e) = s(e)$. 
\end{proof}

By the {\it inverse limit of an inverse system $(\Gamma_i,\phi_{ij})$ of cofinite graphs\/}, we will mean the set $\Gamma=\varprojlim\Gamma_i$ endowed with the unique graph structure and the coarsest uniformity such that the canonical maps $\phi_i\from\Gamma\to\Gamma_i$ are uniformly continuous maps of graphs.

\begin{proposition}\label{Inverse Limit}
Let $(\Gamma_i,\phi_{ij})$ be an inverse system of cofinite graphs. Then the inverse limit $\Gamma=\varprojlim\Gamma_i$ is a cofinite graph. 
\end{proposition}

\begin{proof}
It is easy to see that $\Gamma$ is a Hausdorff cofinite space and a graph as well. So it remains to check that the compatible cofinite entourages of $\Gamma$ form a fundamental system of entourages. Without loss of generality $U = (\bigcap_{n = 1}^N(\pi_{i_n}\times \pi_{i_n})^{-1}[U_{i_n}])\cap \Gamma$, where $U_{i_n}$ is an entourage over $\Gamma_{i_n}$ for all $n$. Then each $\Gamma_{i_n}$, being a cofinite graph, there exists a compatible cofinite entourage $R_{i_n}\subseteq U_{i_n}$ for all $n$. Clearly, $$R =(\bigcap_{n = 1}^N(\pi_{i_n}\times \pi_{i_n})^{-1}[R_{i_n}])\cap \Gamma$$ is compatible cofinite entourage over $\Gamma$ and $R\subseteq U$. Hence our claim that $\Gamma$ is a cofinite graph follows. 
\end{proof}

\begin{note} Here we give an alternative representation of the source, target and edge inversion map.
Let $\Gamma$ be as in the above discussion. Let $x = (x_i)_{i\in I}\in \Gamma$. Let us define the source map $s\from \Gamma\to \Gamma$ via $s(x) = (s(x_i))_{i\in I}$. Clearly, $s$ is well defined and for all $i$ in $I, \phi_i(s(e)) = s_i(e)$, as in the previous lemma. Since each $\phi_i = s_i$  and each $s_i$ is uniformly continuous, we obtain by Corollary~\ref{c:Initial uniformity}$, s\from \Gamma\to \Gamma$ is uniformly continuous. Then, using the uniqueness of $s$, the source map we defined here is equal to the one we defined in the last lemma. Similarly, when convenient we will use $t\from \Gamma\to \Gamma$ as $t(x) = (t(x_i))_{i\in I}$ and $\overline{\phantom e}\from \Gamma\to \Gamma$ as $(\overline x) = (\overline {x_i})_{i\in I}$.  
\end{note}

%%%%%%%%%%%%  {Uniform sum of cofinite graphs}

\subsection{Uniform sum of cofinite graphs}
We now apply the construction in Section~\ref{ss:Sums of cofinite spaces} of uniform sum of finitely many cofinite spaces to finitely many cofinite graphs. 
\begin{proposition}
The uniform sum of a finite family of cofinite graphs is a cofinite graph.
\end{proposition}
\begin{proof}
To begin with, let $(\Gamma_i)_{i\in I}$ be a \textbf{finite} family of cofinite graphs. The  uniform sum of this family $\Gamma=\coprod_{i\in I}\Gamma_i$ has both the structure of a cofinite space and a graph. It only remains to check that $\Gamma$ has a fundamental system of compatible cofinite entourages. Without loss of generality let $U = \bigcup_{i\in I} U_i$ be a cofinite entourage over $\Gamma$. Hence $U_i$ is a cofinite entourage over $\Gamma_i$ for all $i$. But each $\Gamma_i$ is cofinite so there exists a compatible cofinite entourage $R_i\subseteq U_i$. Clearly, $R = \bigcup_{i\in I}R_i$ is a compatible cofinite entourage over $\Gamma$ and $R\subseteq U$. 
\end{proof}

Alternatively, one may define $V(\Gamma)=\coprod_{i\in I}V(\Gamma_i), E(\Gamma)=\coprod_{i\in I}E(\Gamma_i)$. Clearly, $\Gamma = V(\Gamma)\coprod E(\Gamma)$. Also let us define $s\from E(\Gamma)\to V(\Gamma)$ via $s|_{E(\Gamma_i)} = s\from E(\Gamma_i)\to V(\Gamma_i)$. Then, by Lemma~\ref{uniform continuous maps and sum}, $s$ is uniformly continuous, as each restriction $s|_{E(\Gamma_i)}$ is uniformly continuous. Similarly $t, \overline{\phantom e}$ are uniformly continuous as well. Also we make a note of the fact that a uniform sum of uniform spaces also respects their topological structures by being the topological sum of themselves. In particular, each uniform summand is a clopen subgraph of the uniform sum graph.

%%%%%%%%%%%%  {Uniform quotient graphs}

\subsection{Uniform quotient graphs}

Next we apply the construction in Section~\ref{ss:Quotients of cofinite spaces} of uniform quotient spaces to cofinite graphs. 
Let $\Gamma$ be a cofinite graph and let $K$ be a compatible equivalence relation on $\Gamma$. Then the uniform quotient space $\Gamma/\!/K$ of $\Gamma$ modulo $K$ has both the structure of a cofinite space and a graph. We show that these two structures combine to make $\Gamma/\!/K$ into a cofinite graph, provided that it is Hausdorff. 

It remains to say that for each compatible cofinite entourage $R$ of $\Gamma$ with $K\subseteq R$, $(q\times q)[R]$ is compatible, where $q\from \Gamma\to\Gamma/K$ is the quotient map. Let $(K[x],K[y])\in (q\times q)[R]$. This implies that there exists $(u,v)\in R$ such that $q(x) = q(u)$ and $q(y) = q(v)$. Thus $(x,u), (v,y)$ is in $K\subseteq R$. So $(x,y)\in R$. So $(K[x], K[y])$ belongs to $[(q\times q)[R]\cap (V(\Gamma/K)\times V(\Gamma/K))]\dot{\bigcup}[(q\times q)[R]\cap (E(\Gamma/K)\times E(\Gamma/K))]$. Let $(K[e_1],K[e_2])\in (q\times q)[R]$ for $(e_1, e_2)\in E(\Gamma)\times E(\Gamma)$. As $(e_1,e_2)$ is in $R$ and $R$ is a compatible cofinite entourage, we observe that $(s(e_1),s(e_2))$, $(t(e_1),t(e_2))$ and $(\overline{e_1},\overline{e_2})$ are all in $R$. Hence the following $(s(K[e_1]),s(K[e_2])), (t(K[e_1]),t(K[e_2])), (\overline{K[e_1]},\overline{K[e_2]})\in (q\times q)[R]$. Finally, if possible, let $(K[e],\overline {K[e]})\in (q\times q)[R]$. Then as above $(e,\overline e)\in R$, a contradiction. Thus our claim follows.
\begin{proposition}
Let $\Gamma$ be a cofinite graph and $K$ a compatible equivalence relation on $\Gamma$ that satisfies the equivalent  conditions of~\ref{p:Hausdorff quotient}. Then the uniform quotient graph $\Gamma/\!/K$ is a cofinite graph.
\end{proposition}

%%%%%%%%%%Section:  Completion of Cofinite Graphs %%%%%%%%%%

\section{Completions of Cofinite Graphs}

\begin{theorem}\label{t:completion}
Let $\Gamma$ be a cofinite graph contained as a dense subgraph in a compact Hausdorff topological graph $\overline\Gamma$. Then given any compact Hausdorff topological graph $\Delta$ and any uniformly continuous map of graphs $\varphi\from\Gamma\to\Delta$, 
\begin{enumerate}
\item $\overline{V(\Gamma)} = V(\overline {\Gamma})$ and $\overline{E(\Gamma)} = E(\overline{\Gamma})$
\item there exists a unique continuous map of graphs $\overline\varphi\from\overline\Gamma\to\Delta$ extending $\varphi$. 
\end{enumerate}
\end{theorem}
\begin{proof}
\begin{enumerate}
\item Let $v\in V(\overline{\Gamma}), U$ be an open set in $V(\overline{\Gamma})$ containing $v$. Since $\Gamma$ is dense in $\overline{\Gamma}, U\cap\Gamma\neq\emptyset$. Let $w\in U\cap \Gamma$. Since $U\subseteq V(\overline{\Gamma}), w\in V(\Gamma)$. Thus $w\in U\cap  V(\Gamma)$ So $V(\overline{\Gamma}) = \overline{V(\Gamma)}^{V(\overline{\Gamma})}$, the closure of $V(\Gamma)$ in $V(\overline{\Gamma})$. But $\overline{V(\Gamma)}^{V(\overline{\Gamma})} = \overline{V(\Gamma)}$, the closure of $V(\Gamma)$ in $\overline{\Gamma}$. Similarly, $\overline{E(\Gamma)} = E(\overline{\Gamma})$. 

\item Since $\Delta$ is compact, Hausdorff it is a complete uniform space as well. Then there exists a unique uniformly continuous map $\overline {\varphi}\from \overline{\Gamma}\to \Delta$ such that $\overline{\varphi}|_{\Gamma} = \varphi$. So it remains to check that $\overline {\varphi}$ is a map of graphs. 

Let $v\in V(\overline{\Gamma}) = \overline{V(\Gamma)}$. Then there exists a net $\{v_{\alpha}\}_{\alpha\in A}$ in $V(\Gamma)$ such that $\lim_{\alpha\in A}v_{\alpha} = v$. Hence $\overline{\varphi}(v) = \overline{\varphi}(\lim_{\alpha\in A}v_{\alpha}) = \lim_{\alpha\in A}\overline{\varphi}(v_{\alpha}) = \lim_{\alpha\in A}\varphi(v_{\alpha})\in V(\Delta)$, as $\varphi$ is a map of graphs and $V(\Delta)$ is closed in $\Delta$. Similarly, one can show that for all $e$ in $E(\overline{\Gamma}), \overline{\varphi}(e)\in E(\Delta)$.

Let $e\in E(\overline{\Gamma}) = \overline{E(\Gamma)}$. Then there exists a net $\{e_{\alpha}\}_{\alpha\in A}$ in $E(\Gamma)$ such that $\lim_{\alpha\in A}e_{\alpha} = e$. So, $s(\overline{\varphi}(e)) = s(\overline{\varphi}(\lim_{\alpha\in A}e_{\alpha}))$
$= s(\lim\overline{\varphi}(e_{\alpha})) = s(\lim\varphi(e_{\alpha})) = \lim s(\varphi(e_{\alpha})) = \lim\varphi(s(e_{\alpha}))$
$= \lim\overline{\varphi}(s(e_{\alpha})) = \overline{\varphi}(\lim s(e_{\alpha})) = \overline{\varphi}(s(\lim e_{\alpha})) = \overline{\varphi}(s(e))$. Similarly, $t(\overline{\varphi}(e)) = \overline{\varphi}(t(e))$. Now $\overline{\overline{\varphi}(e)} = \overline{\overline{\varphi}(\lim_{\alpha\in A}e_{\alpha})} = \overline{\lim_{\alpha\in A}\overline{\varphi}(e_{\alpha})}$\\
$= \overline{\lim_{\alpha\in A}\varphi(e_{\alpha})} = \lim_{\alpha\in A}\overline{\varphi(e_{\alpha})} = \lim_{\alpha\in A}\varphi(\overline{e_{\alpha}}) = \lim_{\alpha\in A}\overline{\varphi}(\overline{e_{\alpha}})$\\
$= \overline{\varphi}(\lim_{\alpha\in A}\overline{e_{\alpha}}) = \overline{\varphi}(\overline{\lim_{\alpha\in A}e_{\alpha}}) = \overline{\varphi}(\overline e)$. Thus $\overline{\varphi}$ is a map of graphs.    
\end{enumerate}
\end{proof}
\begin{corollary}\label{completion of maps}
As in the previous theorem $\overline{\varphi}(\overline{\Gamma}) = \overline{\varphi(\Gamma)}$.
\end{corollary}
\begin{proof}
The closure of $\Gamma$ is $\overline{\Gamma}$ and $\overline {\varphi}$ is continuous. So $\overline {\varphi}(\overline{\Gamma}) = \overline{\overline {\varphi}(\Gamma)}$\\
$= \overline{\varphi(\Gamma)}$.

On the other hand, since $\overline{\Gamma}$ is compact and $\overline{\varphi}$ is uniformly continuous, $\overline{\varphi}(\overline{\Gamma})$ is a compact subset of the Hausdorff space $\Delta$ and hence is closed. Now $\varphi(\Gamma) = \overline{\varphi}(\Gamma)\subseteq \overline{\varphi}(\overline{\Gamma})$. Thus $\overline{\varphi(\Gamma)}\subseteq \overline{\overline{\varphi}(\overline{\Gamma})} = \overline{\varphi}(\overline{\Gamma})$.  
\end{proof}       
In light of Theorem~\ref{t:completion} we make the following definition. 

\begin{definition}[Completion]\rm
Let $\Gamma$ be a cofinite graph. Then any compact Hausdorff topological graph $\overline\Gamma$ that contains $\Gamma$ as a dense subgraph is called a 
{\it completion\/} of $\Gamma$.
\end{definition}

\begin{corollary}[Uniqueness of completions]
The completion of a cofinite graph $\Gamma$ is unique up to an isomorphism extending the identity map on $\Gamma$.
\end{corollary}
\begin{proof}
If possible, let $\Gamma_i$ be two completions of a cofinite graph $\Gamma$, for $i = 0, 1$. Then the following diagram commutes for unique choices of uniformly continuous maps of graphs $f_{i+1}\from\Gamma_i\to\Gamma_{i+1}$, for $i = 0, 1$ mod $2$,  where $id_{\Gamma}$ is the identity map on $\Gamma$ and $i_{\Gamma_i}$ is the canonical inclusion map,  for $i = 0, 1$. 

$$
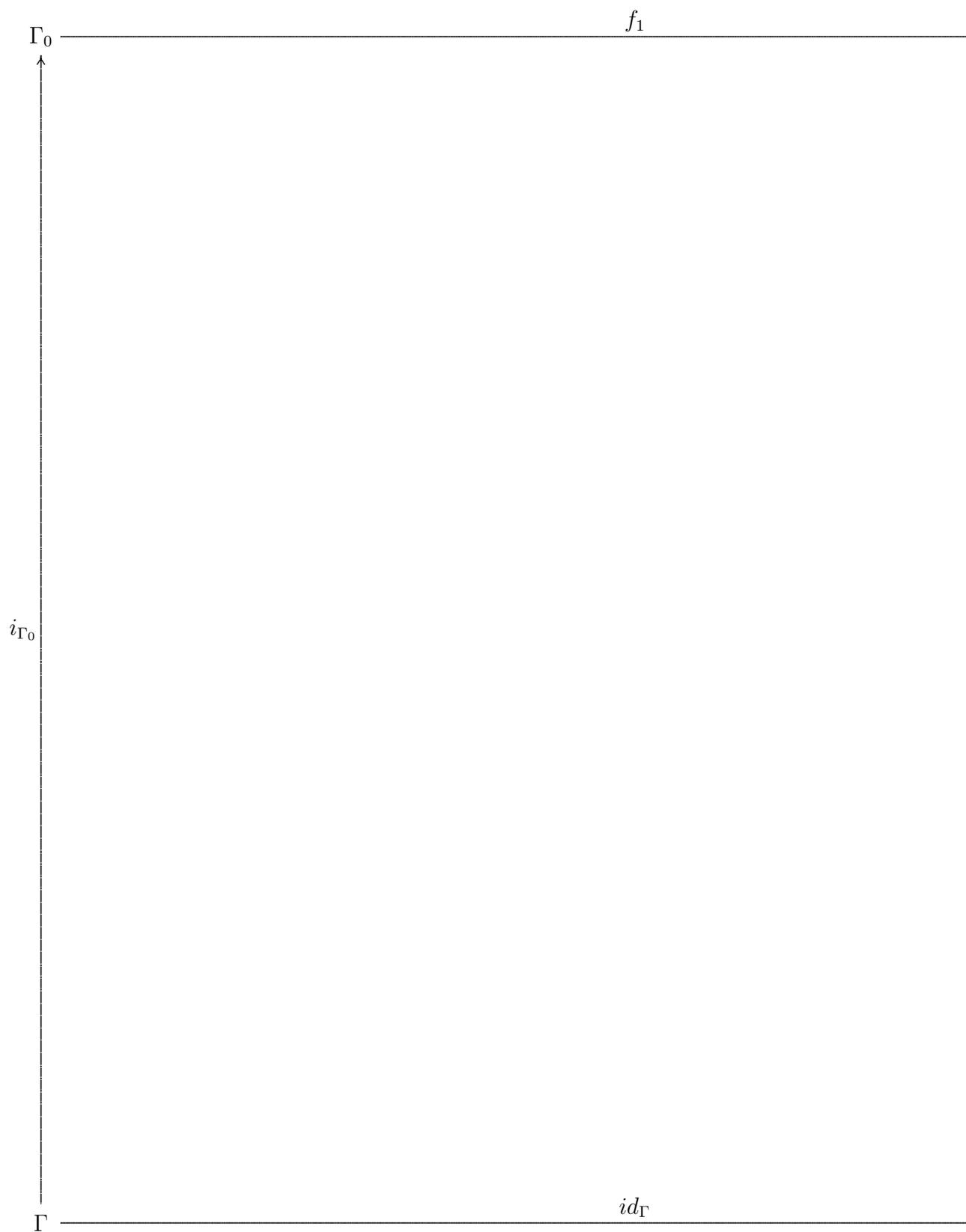
\begindc{\commdiag}[60]
\obj(0,0)[G1]{$\Gamma$}
\obj(10,0)[G2]{$\Gamma$}
\obj(20,0)[G3]{$\Gamma$}
\obj(30,0)[G4]{$\Gamma$}
\obj(0,10)[G5]{$\Gamma_0$}
\obj(10,10)[G6]{$\Gamma_1$}
\obj(20,10)[G7]{$\Gamma_0$}
\obj(30,10)[G8]{$\Gamma_1$}
\mor{G1}{G2}{$id_{\Gamma}$}
\mor{G2}{G3}{$id_{\Gamma}$}
\mor{G3}{G4}{$id_{\Gamma}$}
\mor{G1}{G5}{$i_{\Gamma_0}$}
\mor{G2}{G6}{$i_{\Gamma_1}$}
\mor{G3}{G7}{$i_{\Gamma_0}$}
\mor{G4}{G8}{$i_{\Gamma_1}$}
\mor{G5}{G6}{$f_1$}
\mor{G6}{G7}{$f_0$}
\mor{G7}{G8}{$f_1$}
\enddc$$

But then we also have the following commutative diagrams.
$$
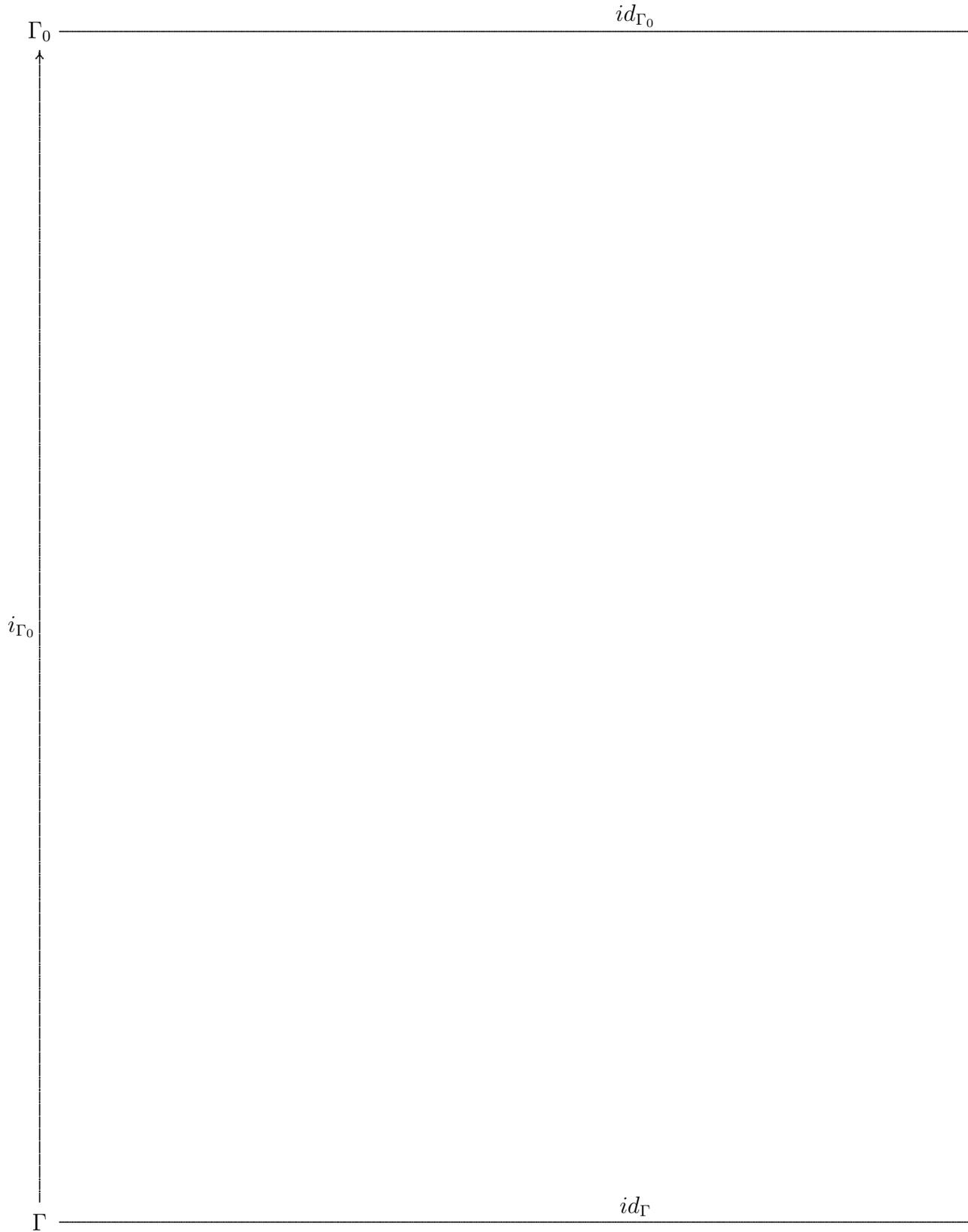
\begindc{\commdiag}[60]
\obj(0,0)[G1]{$\Gamma$}
\obj(10,0)[G2]{$\Gamma$}
\obj(20,0)[G3]{$\Gamma$}
\obj(30,0)[G4]{$\Gamma$}
\obj(0,10)[G5]{$\Gamma_0$}
\obj(10,10)[G6]{$\Gamma_0$}
\obj(20,10)[G7]{$\Gamma_1$}
\obj(30,10)[G8]{$\Gamma_1$}
\mor{G1}{G2}{$id_{\Gamma}$}
\mor{G3}{G4}{$id_{\Gamma}$}
\mor{G1}{G5}{$i_{\Gamma_0}$}
\mor{G2}{G6}{$i_{\Gamma_0}$}
\mor{G3}{G7}{$i_{\Gamma_1}$}
\mor{G4}{G8}{$i_{\Gamma_1}$}
\mor{G5}{G6}{$id_{\Gamma_0}$}
\mor{G7}{G8}{$id_{\Gamma_1}$}
\enddc$$
where $id_{\Gamma_i}$ is the identity map, for $i= 1, 2$.

Thus by Theorem~\ref{t:completion}, $f_0\circ f_1 = id_{\Gamma_0}$ and $f_1\circ f_0 = id_{\Gamma_1}$. Hence $f_0$ and $f_1$ are inverses of each other. 
\end{proof}

\begin{theorem}[Existence of completions]\label{Existence}
Let $\Gamma$ be a cofinite graph and let $I$ be a fundamental system of compatible cofinite entourages of  $\Gamma$, directed by the reverse inclusion. 
Then the inverse limit $\widehat\Gamma=\varprojlim\Gamma/R$ $(R\in I)$ is a compact Hausdorff topological graph and the natural map $\Gamma\to\widehat\Gamma$ embeds $\Gamma$ as a dense subgraph of $\widehat\Gamma$.
\end{theorem}
\begin{proof}
Let us first see that $I$ being a fundamental system of compatible cofinite entourages of  $\Gamma$, directed by the reverse inclusion forms a directed set. This follows as the intersection of two compatible cofinite entourages is also a compatible cofinite entourage.

Let us now see that the uniform quotient graphs $\Gamma/R$ forms an inverse system of finite discrete cofinite graphs, for all $R\in I$. Let $R\leq S$ in $I$. Thus $S\subseteq R$. Let us define $\varphi_{RS}\from\Gamma/S\to\Gamma/R$ via $\varphi_{RS}(S[x]) = R[x]$, for all $x\in \Gamma$. Now, $S[x] = S[y]$ implies that $ (x,y)\in S\subseteq R$ and thus $R[x] = R[y]$. Hence $\varphi_{RS}$ is well defined. Now $S[v]\in V(\Gamma/S)$ implies that $ v\in V(\Gamma)$ so that $ R[v]\in V(\Gamma/R)$. Similarly, if $S[e]\in E(\Gamma/S)$ then we have $R[e]\in E(\Gamma/R)$. Also, $S[e]\in E(\Gamma/S)$ implies that $s(\varphi_{RS}(S[e])) = s(R[e]) = R[s(e)] = \varphi_{RS}(S[s(e)]) = \varphi_{RS}(s(S[e]))$. Similarly, $t(\varphi_{RS}(S[e])) = \varphi_{RS}(t(S[e]))$ and $\overline{\varphi_{RS}(S[e])} = \varphi_{RS}(\overline{S[e]})$. Thus $\varphi_{RS}$ is a map of graphs and since both $\Gamma/S, \Gamma/R$ are discrete, $\varphi_{RS}$ is uniformly continuous as well. Now for $R\leq S\leq T$ in $I, \varphi_{RS}(\varphi_{ST}(T[x])) = \varphi_{RS}(S[x]) = R[x] = \varphi_{RT}(T[x])$, for all $x\in X$. If $R = S$ in $I$, then $\varphi_{RS}(S[x]) = R[x] = S[x] = id_{\Gamma/S}(S[x])$, for all $x\in X$. Hence $(\Gamma/R, \varphi_{RS})_{R\leq S\in I}$ forms an inverse system of discrete cofinite graphs. Hence $\widehat{\Gamma} = \varprojlim_{R\in I}\Gamma/R$ exists. 

Let us now see that $\Gamma$ is densely embedded in $\widehat{\Gamma}$. Let $\varphi_R\from \widehat{\Gamma}\to \Gamma/R$ be the corresponding canonical projection map and let $\eta_R\from\Gamma\to\Gamma/R$ be the canonical surjection for all $R$ in $I$. Then the following diagram commutes for all $R\leq S$ in $I$, as $\varphi_{RS}(\eta_S(\gamma)) = \varphi_{RS}(S[\gamma]) = R[\gamma] = \eta_R(\gamma)$, for all $\gamma\in\Gamma$.
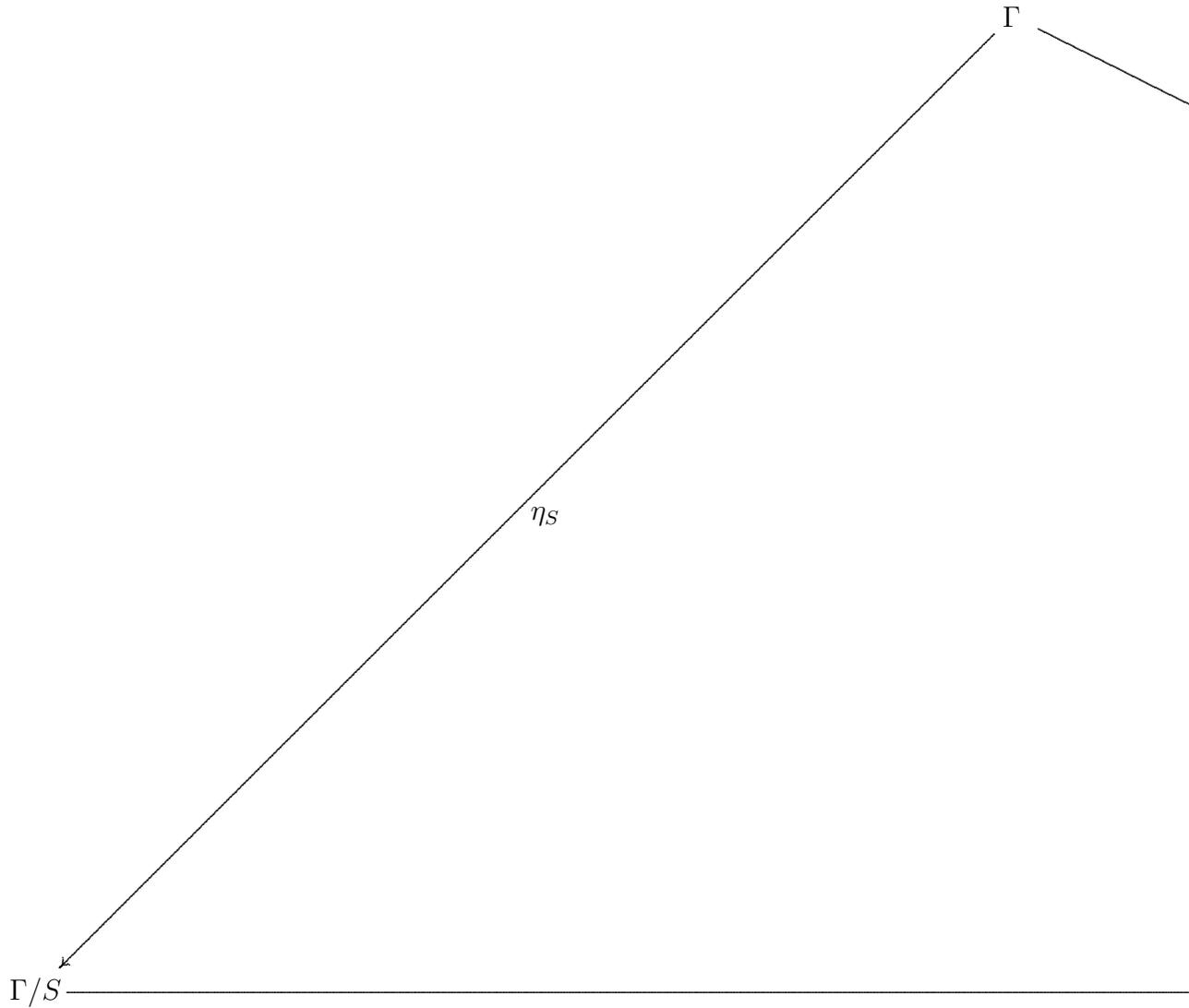
$$\begindc{\commdiag}[40]
\obj(10,10)[G1]{$\Gamma$}
\obj(0,0)[G2]{$\Gamma/S$}
\obj(30,0)[G3]{$\Gamma/R$}
\mor{G1}{G2}{$\eta_S$}
\mor{G2}{G3}{$\varphi_{RS}$}
\mor{G1}{G3}{$\eta_R$}
\enddc$$
Hence $(\Gamma,\eta_R)_{R\in I}$ forms a compatible system to the aforesaid inverse system of cofinite graphs. Thus there exists a uniformly continuous map of graphs $\theta\from\Gamma\to\widehat{\Gamma}$ such that the following diagram commutes for all $R$ in $I$.
$$\begindc{\commdiag}[60]
\obj(10,10)[G1]{$\Gamma$}
\obj(0,0)[G2]{$\widehat{\Gamma} = \varprojlim_{R\in I}$}
\obj(30,0)[G3]{$\Gamma/R$}
\mor{G1}{G2}{$\theta$}
\mor{G2}{G3}{$\varphi_R$}
\mor{G1}{G3}{$\eta_R$}
\enddc$$ 
Now let $x_1,x_2\in\widehat{\Gamma}$ be such that $\theta(x_1) = \theta(x_2)$. So for all $R$ in $I$ we get $R[x_1] = \eta_R(x_1) = \varphi_R(\theta(x_1)) = \varphi_R(\theta(x_2))  = \eta_R(x_2) = R[x_2]$. Thus $(x_1,x_2)\in \bigcap_{R\in I}R = D(\Gamma)$, as $\Gamma$ is Hausdorff. Hence $x_1 = x_2$. So $\theta$ is injective. 
So it remains to check that $\theta$ is a topological embedding. This follows from the claim that $\theta(R[x]) = \varphi_R^{-1}(\eta_R(x))\cap\theta(\Gamma)$, for all $x\in\Gamma$ and for all $R\in I$. The above claim follows as $p\in \theta(R[x])\Leftrightarrow$ there exists $q\in R[x]\cap \Gamma$ such that $\theta(q) = p\Leftrightarrow$ there exists $q\in\Gamma$ such that $\eta_R(q) = \eta_R(x)$ and $\theta(q) = p\Leftrightarrow$ there exists, $q\in\Gamma$ such that $\varphi_R(\theta(q)) = \eta_R(x)$ and $\theta(q) = p\Leftrightarrow \varphi_R(p) = \eta_R(x)\Leftrightarrow p\in \varphi_R^{-1}(\eta_R(x))\cap\theta(\Gamma)$.
\end{proof}
Notice that in the definition of the completion of $\Gamma$, we did not insist that $\overline\Gamma$ be a cofinite graph. However, it turns out that this will automatically be so. 
To see this, we first prove a lemma.

\begin{lemma}\label{pro entourage}
Let $\overline\Gamma$ be the completion of a cofinite graph $\Gamma$ and let $R$ be a compatible cofinite entourage of $\Gamma$. Then $\overline R$ is a compatible cofinite entourage of $\overline\Gamma$ and $\overline R\cap(\Gamma\times\Gamma)=R$.
\end{lemma}

\begin{proof}
The quotient $\Gamma/R$ is a compact Hausdorff topological graph and the quotient map $\eta_R\from\Gamma\to\Gamma/R$ is uniformly continuous. 
So by Theorem~\ref{t:completion}, $\eta_R$ extends to a continuous map of graphs 
$\overline\eta_R\from\overline\Gamma\to\Gamma/R$. Using Corollary~\ref{completion of maps} and as $\eta_R$ is surjective, $\overline\eta_R(\overline{\Gamma}) = \overline{\eta_R(\Gamma)} = \overline{\Gamma/R} = \Gamma/R$. Thus $\overline\eta_R$ is surjective as well. Since $D(\Gamma/R)$ is a compatible cofinite entourage over $\Gamma/R$ and $(\overline\eta_R \times \overline\eta_R)^{-1}[D(\Gamma/R)] = \overline\eta_R^{-1}\overline\eta_R$ we see that $\overline\eta_R^{-1}\overline\eta_R$ is a compatible cofinite equivalence relation over $\overline{\Gamma}$ and thus endowed with the quotient topology $\overline{\Gamma}/\overline\eta_R^{-1}\overline\eta_R$ is a discrete quotient graph of $\overline{\Gamma}$ and we claim that the map $\overline\eta_R$determines an isomorphism of topological graphs $\Psi\from\overline\Gamma/\overline\eta_R^{-1}\overline\eta_R\to\Gamma/R$. Let us define $\Psi(\overline\eta_R^{-1}\overline\eta_R[x]) = \overline\eta_R[x]$ for all $x$ in $\overline{\Gamma}$. If $\overline\eta_R^{-1}\overline\eta_R[x] = \overline\eta_R^{-1}\overline\eta_R[y]$ then $(x,y)\in \overline\eta_R^{-1}\overline\eta_R$ so that $\overline\eta_R(x) = \overline\eta_R(y)$. Hence $\Psi$ is a well defined injection. As $\overline\eta_R$ is a surjective map of graphs so is $\Psi$. Since both $\overline\Gamma/\overline\eta_R^{-1}\overline\eta_R, \Gamma/R$ are discrete topological graphs, both $\Psi , \Psi^{-1}$ are uniformly continuous and our claim that $\Psi$  is an isomorphism of  topological graphs follows. 

Since $\overline\eta_R^{-1}\overline\eta_R\cap(\Gamma\times\Gamma)=\eta_R^{-1}\eta_R=R$. 
It now suffices to show that $\overline\eta_R^{-1}\overline\eta_R=\overline R$. First note that $R=\eta_R^{-1}\eta_R\subset\overline\eta_R^{-1}\overline\eta_R$ and that $\overline\eta_R^{-1}\overline\eta_R$ is closed in $\overline\Gamma\times\overline\Gamma$ as $\Gamma/R$ is finite and discrete and thus $D(\Gamma/R)$ is a clopen subset of $\Gamma/R\times \Gamma/R$;
whence $\overline R\subseteq\overline\eta_R^{-1}\overline\eta_R$.
Conversely, let $z\in\overline\eta_R^{-1}\overline\eta_R$ and let $V$ be a neighborhood of $z$ in $\overline\Gamma\times\overline\Gamma$. 
Then $V\cap\overline\eta_R^{-1}\overline\eta_R$ is also a neighborhood of~$z$. 
However, $\Gamma\times\Gamma$ is dense in $\overline\Gamma\times\overline\Gamma$, so  we ca say that $V\cap R=V\cap\overline\eta_R^{-1}\overline\eta_R\cap(\Gamma\times\Gamma) \ne\emptyset.
$. Therefore $z\in\overline R$ and $\overline\eta_R^{-1}\overline\eta_R\subseteq\overline R$. Thus the claim.
\end{proof}

\begin{theorem} \label{t:Completion is cofinite}
Let $\Gamma$ be a cofinite graph and let $I$ be the filter base of all compatible cofinite entourages of $\Gamma$. 
Then the completion $\overline\Gamma$ is also a cofinite graph and 
$\{\overline R \mid R\in I\}$ is the filter base of all compatible cofinite entourages of $\overline\Gamma$.
\end{theorem}
\begin{proof}
We will first see that $\{\overline R \mid R\in I\}$ forms the filter base of all compatible cofinite entourages of $\overline\Gamma$. For let $\overline R, \overline S$ be the compatible cofinite entourages over $\overline{\Gamma}$ for $R, S$ in $I$. Then there is $T\in I$ such that $T\subseteq R\cap S$. Now $\overline T\subseteq \overline{R\cap S}\subseteq \overline R\cap \overline S$. Now let $K$ be any compatible cofinite entourage over $\overline{\Gamma}$. Then $K\cap(\Gamma\times \Gamma)$ is a compatible cofinite entourage over $\Gamma$. Hence there exists some $R$ in $I$, such that $R = K\cap(\Gamma\times \Gamma)$. Since $K$ is open in $\overline{\Gamma}\times \overline{\Gamma}$, any open set $U$ in $K$ is also open in $\overline{\Gamma}\times\overline{\Gamma}$. Now for all $(x,y)\in K$ and $U\in\eta_(x,y)$ in $K$, $U\cap (\Gamma\times \Gamma)\ne\emptyset$ as $\Gamma\times \Gamma$ is dense in $\overline{\Gamma}\times \overline{\Gamma}$. Hence $U\cap (K\cap(\Gamma\times \Gamma)) = U\cap R\ne\emptyset$ and thus $R$ is dense in $K$. It follows that $\overline R = \overline K = K$. Hence $\{\overline R \mid R\in I\}$ forms the filter base of all compatible cofinite entourages over $\overline{\Gamma}$. It remains to show that $\{\overline R\mid R\in I\}$ is a fundamental system of entourages of $\overline{\Gamma}$. For this purpose let $W$ be any entourage of $\overline{\Gamma}$. We may assume that $W$ is closed in $\overline{\Gamma}\times\overline{\Gamma}$, as the closed entourages form a fundamental system of entourages. Since $W\cap(\Gamma\times\Gamma)$ is an entourage of $\Gamma$ and $\Gamma$ is a cofinite graph, there exists $R\in I$ such that $R\subseteq W\cap(\Gamma\times\Gamma)$. Now $\overline R\subseteq \overline W = W$ and we see that every entourage of $\overline{\Gamma}$ contains a member of the set $\{\overline R\mid R\in I\}$, as required.
\end{proof}

It follows from Theorem~\ref{t:Completion is cofinite} that the completion of a cofinite graph is a profinite graph.

\bibliographystyle{amsplain}

\end{document}